\documentclass[12pt,reqno]{amsart}
\usepackage{amsaddr}
\usepackage{amssymb,latexsym,amsmath,epsfig,amsthm,mathrsfs}
\usepackage{rotating}
\usepackage{graphicx}
\usepackage{amssymb}
\usepackage{lineno}
\usepackage{enumitem}
\usepackage{cite}
\usepackage[top=3cm, bottom=3cm, left=3cm, right=3cm]{geometry}
\usepackage[usenames]{color}
\usepackage[colorlinks=true,
linkcolor=blue,
filecolor=blue,
citecolor=blue]{hyperref}

\newtheorem{thma}{Theorem}

\newtheorem{lemmaa}{Lemma}

\newtheorem{theorem}{Theorem}[section]

\newtheorem{lemma}[theorem]{Lemma}
\newtheorem{conjecture}{Conjecture}

\theoremstyle{definition}

\theoremstyle{remark}

\theoremstyle{definition}

\numberwithin{equation}{section}

\begin{document}
	
	\title[Restricted Signed Sumsets - I]{Direct and Inverse Problems for Restricted Signed Sumsets - I}
	

\author[R K Mistri]{Raj Kumar Mistri$^{*}$}
\address{\em{\small Department of Mathematics, Indian Institute of Technology Bhilai\\
Durg, Chhattisgarh, India\\
e-mail: rkmistri@iitbhilai.ac.in}}


\author[N Prajapati]{Nitesh Prajapati$^{**}$}
\address{\em{\small Department of Mathematics, Indian Institute of Technology Bhilai\\
Durg, Chhattisgarh, India\\
email: niteshp@iitbhilai.ac.in}}

\thanks{$^{*}$Corresponding author}
\thanks{$^{**}$The research of the author is supported by the UGC Fellowship (NTA Ref. No.: 211610023414)}

	\subjclass[2010]{Primary 11P70; Secondary 11B13, 11B75}

	\keywords{sumsets, restricted sumsets, signed sumset, restricted signed sumset.}

\begin{abstract}
Let $A=\{a_{1},\ldots,a_{k}\}$ be a nonempty finite subset of an additive abelian group $G$. For a positive integer $h$, the {\em $h$-fold signed sumset of $A$}, denoted by $h_{\pm}A$, is defined as
	$$h_{\pm}A=\left\lbrace \sum_{i=1}^{k} \lambda_{i} a_{i}: \lambda_{i} \in \{-h, \ldots, 0, \ldots, h\} \ \text{for} \ i=  1, 2, \ldots, k \ \text{and} \  \sum_{i=1}^{k} \left|\lambda_{i} \right| =h\right\rbrace, $$
and the {\em restricted $h$-fold signed sumset of $A$}, denoted by $h^{\wedge}_{\pm}A$, is defined as
	$$h^{\wedge}_{\pm}A=\left\lbrace \sum_{i=1}^{k} \lambda_{i} a_{i}: \lambda_{i} \in \left\lbrace -1, 0, 1\right\rbrace \ \text{for} \ i=  1, 2, \ldots, k \ \text{and} \  \sum_{i=1}^{k} \left|\lambda_{i} \right| = h\right\rbrace. $$
A direct problem for the sumset $h^{\wedge}_{\pm}A$ is to find the optimal size of $h^{\wedge}_{\pm}A$ in terms of $h$ and $|A|$. An inverse problem for this sumset is to determine the structure of the underlying set $A$ when the sumset $h^{\wedge}_{\pm}A$ has optimal size. While some results are known for the signed sumsets in finite abelian groups due to Bajnok and Matzke, not much is known for the restricted $h$-fold signed sumset $h^{\wedge}_{\pm}A$ even in the additive group of integers $\Bbb Z$. In case of $G = \Bbb Z$, Bhanja, Komatsu and Pandey studied these problems for the sumset $h^{\wedge}_{\pm}A$ for $h=2, 3$, and $k$, and conjectured the direct and inverse results for $h \geq 4$. In this paper, we prove these conjectures completely for the sets of positive integers. In a subsequent paper, we prove these conjectures for the sets of nonnegative integers.
\end{abstract}

	\maketitle
	

\section{Introduction}
Let $\Bbb Z$ denote the set of integers. For  integers $a$ and $b$ with $a \leq b$, we denote the set $\{n \in \Bbb Z: a \leq n \leq b\}$ by $[a, b]$. Let $|S|$ denote the size of the finite set $S$. An {\em arithmetic progression (A.P.)} of integers with $k$-terms and common difference $d$ is a set $A$ of the form $\{a + id: i = 0, 1, \ldots, k-1\}$, where $a, d \in \Bbb Z$ and $d \neq 0$. Let $G$ be an additive abelian group, and let $A_1, \ldots, A_h$ be nonempty finite subsets of $G$, where $h$ is a positive integer. The {\em sumset} of $A_1, \ldots, A_h$, denoted by $A_1 + \cdots + A_h$, is defined as
\[A_1 + \cdots + A_h := \{a_1 + \cdots + a_h: a_i \in A_i ~\text{for}~ i = 1, \ldots, h\}.\]
The {\em restricted sumset} of $A_1, \ldots, A_h$, denoted by $A_1 \dotplus \cdots \dotplus A_h$, is defined as
\[A_1 \dotplus \cdots \dotplus A_h := \{a_1 + \cdots + a_h: a_i \in A_i ~\text{for}~ i = 1, \ldots, h~ \text{and}~ a_i \neq a_j~\text{for}~i \neq j\}.\]
If $A_i = A$ for $i = 1, \ldots, h$, then the sumset $\underbrace{A + \cdots + A}_{h~\text{copies}}$ is usually denoted by $hA$, and it is called the {\em $h$-fold sumset of the set $A$}. Similarly, the restricted sumset $\underbrace{A \dotplus \cdots \dotplus A}_{h~\text{copies}}$ is usually denoted by $h^{\wedge}A$, and it is called the {\em restricted $h$-fold sumset of the set $A$}. Thus if $A = \{a_1, \ldots, a_k\}$, then
\[hA = \left\{\sum_{i=1}^{k}\lambda_i a_i: \lambda_i \in [0, h] ~\text{for}~ i = 1, \ldots, k ~\text{and}~ \sum_{i=1}^{k}\lambda_i = h\right\},\]
and
\[h^{\wedge}A =  \left\{\sum_{i=1}^{k}\lambda_i a_i: \lambda_i \in [0, 1] ~\text{for}~ i = 1,\ldots, k ~\text{and}~ \sum_{i=1}^{k}\lambda_i = h \right\}.\]

The study of the sumsets dates back to Cauchy \cite{cauchy} who proved that if $A$ and $B$ are nonempty subsets of $\Bbb Z_p$, then $|A + B| \geq \min(p, |A| + |B| -1)$, where $\Bbb Z_p$ is the group of prime order $p$. The result has been known as {\em Cauchy-Davenport Theorem} after Davenport rediscovered this result \cite{dav1, dav2} in $1935$. The Cauchy-Davenport theorem immediately implies that if $A$ is a nonempty subsets of $\Bbb Z_p$, then $|hA| \geq \min(p, h|A| - h + 1)$ for all positive integers $h$. The corresponding theorem for the restricted $h$-fold sumset $h^{\wedge}A$ in $\Bbb Z_p$ is due to Dias da Silva and Hamidoune (see \cite{dias}) who proved using the theory of exterior algebra that if $A$ is a nonempty subsets of $\Bbb Z_p$, then $|h^{\wedge}A| \geq \min(p, h|A| - h^2 + 1)$ for all positive integers $h \leq |A|$. Later this result was reproved by Alon, Nathanson and Ruzsa (see \cite{alon1} and \cite{alon2}) by means of {\em polynomial method} which is a very powerful tool for tackling certain problems in additive combinatorics. The theorem for $h = 2$ is known as {\em Erd\H{o}s - Heilbronn conjecture} which was first conjectured by Erd\H{o}s and Heilbronn (see \cite{erdos}) during $1960$.
	
The following theorem provides the optimal lower bound for the size of restricted $h$-fold sumset $h^{\wedge}A$ in the additive group of integers $\Bbb Z$.
\begin{thma}[{\cite[Theorem 1]{nath1}; \cite[Theorem 1.9]{nath}}]\label{restricted-hfold-direct-thm}
Let $A$ be a nonempty finite set of integers with $|A| = k$ and let $1 \leq h \leq k.$ Then
\begin{equation}\label{res-sum-bound}
  |h^{\wedge}A| \geq hk - h^2 + 1.
\end{equation}
The lower bound in $(\ref{res-sum-bound})$ is best possible.
\end{thma}

Next theorem characterizes the sets $A \subseteq \Bbb Z$ for which the equality holds in $(\ref{res-sum-bound})$.
\begin{thma}[{ \cite[Theorem 2]{nath1}; \cite[Theorem 1.10]{nath}}]\label{restricted-hfold-inverse-thm}
Let $k \geq 5$ and let $2 \leq h \leq k - 2.$ If $A$ is a set of $k$ integers such that
\[|h^{\wedge}A| = hk - h^2 + 1,\]
then $A$ is a $k$-term arithmetic progression.
\end{thma}

These sumsets and other kind of sumsets have been studied extensively in liteartures (see \cite{nath, tao, freiman, mann} and the references given therein).

Two other variants of these sumsets have appeared recently in the literature in the works of Bajnok, Ruzsa and other researchers \cite{bajnok-ruzsa2003, bajnok-matzke2015,bajnok-matzke2016, bajnok2018, bhanja-pandey2019, bhanja-kom-pandey2021, klopsch-lev2003, klopsch-lev2009}: the {\it $h$-fold signed sumset} $h_{\pm}A$ and the {\it restricted $h$-fold signed sumset} $h^{\wedge}_{\pm}A$ of the set $A$ which are defined as follows:
	\begin{equation*}
		h_{\pm}A := \left\{\sum_{i=1}^{k} \lambda_i a_i : \lambda_i \in [-h, h] ~\text{for}~ i = 1, 2, \ldots ,k ~\text{and}~ \sum_{i=1}^{k} |\lambda_i| =h \right\},
	\end{equation*}
	and
	\begin{equation*}
		h^{\wedge}_{\pm}A := \left\{\sum_{i=1}^{k} \lambda_i a_i : \lambda_i \in [-1, 1] ~\text{for}~ i = 1, 2, \ldots ,k ~\text{and}~ \sum_{i=1}^{k} |\lambda_i| =h \right\}.
	\end{equation*}
It is easy to see that
\[h^{\wedge}A \cup h^{\wedge}(-A) \subseteq h_{\pm}^{\wedge}A \subseteq h_{\pm}^{\wedge}(A \cup -A),\]
and
\[h\hat{}_{\pm}A \subseteq h_{\pm}A,\]
For a nonzero integer $c$, we have
\[h^{\wedge}_{\pm}(c \ast A) = c \ast (h^{\wedge}_{\pm}A),\]
where $c \ast A = \{ca: a \in A\}$ and $-A = (-1) \ast A $.

While $h$-fold sumsets are well-studied in the literature, the $h$-fold signed sumsets are not well-studied in the literature. The signed sumsets appear naturally in the literature in many contexts. The $h$-fold signed sumset $h_{\pm}A$ first appeared in the work of Bajnok and Ruzsa \cite{bajnok-ruzsa2003} who studied it in the context of the independence number of a subset of an abelian group $G$ (see also \cite{bajnok2000} and \cite{bajnok2004}), and it also appeared in the work of Klopsch and Lev \cite{klopsch-lev2003,klopsch-lev2009} in the context of dimameter of the group $G$ with respect to the set $A$. The {\em independence number} of a subset $A$ of $G$ is defined \cite{bajnok-ruzsa2003} as the largest positive integer $t$ such that
\[0 \not \in \bigcup\limits_{h=1}^{t} h_{\pm}A.\] 
The {\em diameter} of $G$ with respect to $A$ is defined \cite{klopsch-lev2003} as the smallest positive integer $s$ such that 
\[\bigcup_{h=0}^{s} h_{\pm}A=G.\]

For a positive integer $m \leq |G|$, define
\[\rho(G, m, h) = \min \{|hA| : A \subseteq G, |A|=m\}\] 
and   
\[\rho_{\pm}(G, m, h) = \min \{|h_{\pm}A| : A \subseteq G, |A|=m\}.\] 
Bajnok and Matzke initiated the detailed study of the function $\rho_{\pm}(G, m, h)$, and they proved that $\rho_{\pm}(G, m, h) = \rho(G, m, h)$, when $G$ is a finite cyclic group \cite{bajnok-matzke2015}. In another work, they studied the cases when $\rho_{\pm}(G, m, h) = \rho(G, m, h)$, where $G$ is an elementary abelian group \cite{bajnok-matzke2016}. In a recent paper \cite{bhanja-pandey2019}, Bhanja and Pandey have studied the direct and inverse problems in the additive group $\mathbb{Z}$ of integers. They obtained the optimal lower bound for the cardinality of the sumset $h_{\pm}A$. They also proved that if the optimal lower bound is achieved, then $A$ must be a certain arithmetic progression.

In case of restricted signed sumset $h^{\wedge}_{\pm}A$, not much is known even in the additive group of integers $\Bbb Z$. The direct problem for the sumset $h_{\pm}^{\wedge}A$ is to find lower bounds for $|h_{\pm}^{\wedge}A|$ in terms of $|A|$. The inverse problem for this sumset is to determine the structure of the finite sets $A$ of for which $|h_{\pm}^{\wedge}A|$ is optimal. In this direction, recently, Bhanja, Komatsu and Pandey \cite{bhanja-kom-pandey2021} solved some cases of both the direct and inverse problems for  $h^{\wedge}_{\pm}A$ in $\Bbb Z$ and conjectured for the rest of the cases. More precisely, they proved the following result.
	
\begin{thma}[{\cite[Theorem 2.1]{bhanja-kom-pandey2021}}]\label{thm:3}
Let $h$ and $k$ be positive integers with $h \leq k$. Let $A$ be a set of $k$ positive integers. Then
\begin{equation}\label{Bound A}
	\left|h^{\wedge}_{\pm}A\right| \geq   2(hk-h^2)+ \frac{h(h+1)}{2} + 1.
\end{equation}
These lower bounds are best possible for $h=1, 2$ and $k$.
\end{thma}

\begin{thma}[{\cite[Theorem 3.1]{bhanja-kom-pandey2021}}]\label{thm:5}
		Let $h$ and $k$ be integer such that $1 \leq h \leq k$. Let $A$ be set of $k$ nonnegative integers such that $0 \in A$. Then
\begin{equation}\label{Bound B}
  |h_{\pm}^\wedge A| \geq 2(hk - h^2) + \frac{h(h - 1)}{2} + 1.
\end{equation} 
		This lower bound is best possible for $h = 1, 2$, and $k$. 
\end{thma}

\begin{thma}[{\cite[Theorem 2.3]{bhanja-kom-pandey2021}}]\label{thm:4}
Let $h \geq 3$ be a positive integer. Let $A$ be the set of $h$ positive integers such that $|h_{\pm}^\wedge A| = \frac{h(h + 1)}{2} + 1$. Then 
		\begin{equation}
			A =
			\begin{cases}
				\{a_1, a_2, a_1 + a_2\} ~\text{with}~ 0 < a_1 < a_2, & \mbox{if } h = 3;\\
				d \ast [1, h] ~\text{for some positive integer}~ d, & \mbox{if } h \geq 4.  
			\end{cases}
		\end{equation}
\end{thma}

\begin{thma}[{\cite[Theorem 3.3]{bhanja-kom-pandey2021}}]\label{thm:6}
		Let $h \geq 4$ be a positive integer. Let $A$ be the set of $h$ nonnegative integers with $0 \in A$ such that $|h_{\pm}^\wedge A| = \frac{h(h - 1)}{2} + 1$. Then 
		\begin{equation}
			A =
			\begin{cases}
				\{0, a_1, a_2, a_1 + a_2\} ~\text{with}~ 0 < a_1 < a_2, &\text{if $h = 4$};\\
				d \ast [0, h - 1] ~\text{for some positive integer}~ d, &\text{if $h \geq 5$.}
			\end{cases}
		\end{equation}
\end{thma}

In the same paper, they proved the inverse theorems for $|2_{\pm}^\wedge A|$ also (see \cite[Theorem 2.2, Theorem 2.3, Theorem 3.2, and Theorem 3.3]{bhanja-kom-pandey2021}). It can be verified that the lower bounds in (\ref{Bound A}) is not optimal for $3 \leq h \leq k-1$. For these cases, they conjectured the lower bounds and the inverse results, and proved these conjectures for the case $h = 3$ (see \cite[Theorem 2.5 and Theorem 3.5]{bhanja-kom-pandey2021})). The precise statements of the conjectures are the following:
	
	\begin{conjecture}[{\cite[Conjecture 2.4, Conjecture 2.6]{bhanja-kom-pandey2021}}]\label{Conjecture 1}			
		Let $A$ be a set of $k \geq 4$ positive integers, and let $h$ be an integer with $3 \leq h \leq k-1$. Then
		\begin{equation*}\label{Lower bound Conjecture 1}
			\left|h^{\wedge}_{\pm}A\right| \geq 2hk-h^2 + 1.
		\end{equation*}
		This lower bound is best possible.

 Moreover,  if $|h^{\wedge}_{\pm}A| = 2hk-h^2 + 1,$ then $A = d \ast \{1,3,\ldots, 2k-1\}$ for some positive integer $d$.
	\end{conjecture}
	
	\begin{conjecture}[{\cite[Conjecture 3.4, Conjecture 3.7]{bhanja-kom-pandey2021}}]\label{Conjecture 2}
		Let $A$ be a set of $k \geq 5$ nonnegative integers with $0 \in A$, and let $h$ be an integer with $3 \leq h \leq k-1$. Then
		\begin{equation*}
			\left|h^{\wedge}_{\pm}A\right| \geq 2hk - h(h+1) + 1.
		\end{equation*}
		This lower bound is best possible.

 Moreover, if $\left|h^{\wedge}_{\pm}A\right| = 2hk-h(h+1) + 1,$ then $A = d \ast [0,k-1]$ for some positive integer $d$.
	\end{conjecture}
	
Mohan, Mistri and Pandey confirmed the conjecture for $h =4$, and they also proved the conjectures for certain special types of sets, including arithmetic progression \cite{mmp2024}. In this paper, we prove Conjecture \ref{Conjecture 1} as following two theorems. The proof of Conjecture \ref{Conjecture 2} requires some more works. Hence we prove Conjecture \ref{Conjecture 2} in a subsequent paper \cite{mistri-prajapati2025b}.
	
\begin{theorem}\label{rssp-thm-1}
		Let $h$ and $k$ be positive integers such that $3 \leq h \leq k - 1$. Let $A = \{a_1, \ldots, a_k\}$ be a set of positive integers such that $a_1 < \cdots < a_k$. Then
\begin{equation}\label{rssp-thm-1eq1}
|h_{\pm}^\wedge A| \geq 2hk - h^2 + 1.
\end{equation}
The lower bound in $\eqref{rssp-thm-1eq1}$ is best possible.
\end{theorem}
	
\begin{theorem}\label{rssp-thm-2}
		Let $h$ and $k$ be positive integers such that $3 \leq h \leq k - 1$. Let $A = \{a_1, \ldots, a_k\}$ be a set of positive integers such that $a_1 < \cdots < a_k$. If
		\[|h_{\pm}^\wedge A| = 2hk - h^2 + 1,\]
	    then 
		\[A = a_1 \ast \{1, 3, \ldots, 2k - 1\}.\]
\end{theorem}

For a set $A \subseteq \Bbb Z$, let $A_{abs} = \{|a|: a \in A\}$. It is easy to verify that if $A$ is a nonempty finite set of integers such that either $A \cap (-A) = \emptyset$ or $A \cap (-A) = \{0\}$, then
\[h_{\pm}^{\wedge} A = h_{\pm}^{\wedge} A_{abs}.\]

This identity and above theorems immediately imply the following theorems.
\begin{theorem}\label{rssp-thm-1a}
		Let $h$ and $k$ be positive integers such that $3 \leq h \leq k - 1$. Let $A$ be a set of $k$ integers such that $A \cap (-A) = \emptyset$. Then
		\[|h_{\pm}^\wedge A| \geq 2hk - h^2 + 1.\]
		This lower bound is best possible.
\end{theorem}
	
\begin{theorem}\label{rssp-thm-2a}
		Let $h$ and $k$ be positive integers such that $3 \leq h \leq k - 1$. Let $A$ be a set of $k$ integers such that $A \cap (-A) = \emptyset$. If
		\[|h_{\pm}^\wedge A| = 2hk - h^2 + 1,\]
	    then 
		\[A_{abs} = d \ast \{1, 3, \ldots, 2k - 1\},\]
        where $d$ is the smallest element of $A_{abs}$. 
\end{theorem} 

Mohan, Mistri and Pandey proved the following lemma for the restricted signed sumset.
\begin{lemmaa}[{\cite[Lemma $1$]{mmp2024}}]\label{rssp-basic-lem1}
Let $h$ and $k$ be integers such that $3 \leq h \leq k - 1$. Let $A = \{a_1,\ldots, a_k\}$ be a set of integers such that $a_1 < \cdots < a_k$. Let $B = \{a_1, \ldots, a_{h + 1}\} \subseteq A$. If $|h_{\pm}^\wedge B| \geq h^2 + 2h + 1 + t$, where $t \geq 0$, then 
\[|h_{\pm}^\wedge A| \geq 2hk - h^2 + 1 + t.\]
\end{lemmaa} 
In view of the above lemma, to prove Theorem \ref{rssp-thm-1}, it suffices to prove the following theorem.
\begin{theorem}\label{rssp-thm-3}
		Let $h$ be an integers such that $h \geq 3$. Let $A = \{a_1, \ldots, a_{h + 1}\}$ be a set of positive integers such that $a_1 < \cdots < a_{h + 1}$. Then
\begin{equation}\label{rssp-thm-3eq1}
 |h_{\pm}^\wedge A| \geq h^2 + 2h + 1.
\end{equation}
The lower bound in $\eqref{rssp-thm-3eq1}$ is best possible.
\end{theorem}

In Section \ref{section-rssp-aux-lemma}, we prove some auxiliary lemmas which will be used to prove Theorem \ref{rssp-thm-3} (hence Theorem \ref{rssp-thm-1}) and Theorem \ref{rssp-thm-2} in Section \ref{section-rrsp-thm-proof}.

\section{Auxiliary lemmas} \label{section-rssp-aux-lemma}
For a nonempty finite set $A$ of integers, let $\min(A)$, $\max(A)$, $\min_{+}(A)$, $\max_{-}(A)$ denote the smallest, the largest, the second smallest, and the second largest elements of $A$, respectively.  A set $S$ is said to be {\it symmetric} if $x \in S$ implies $-x \in S$. For a subset $A$ of an additive abelian group $G$, if $c \in G$, then we write $c + A$ for $\{c\} + A$. The {\em set of subsums of $A \subseteq G$}, denoted by $\Sigma (A)$ is defined as follows.  
\[\Sigma (A) = \bigg\{\sum_{b \in B}b : B \subseteq A\bigg\}.\]

The following facts will be used frequently in the proofs of lemmas.
\begin{enumerate}
	\item Let $h \geq 3$ be an integer. Let $A = \{a_1,\ldots, a_{h + 1}\}$ be a set of integers such that $a_1 < \cdots < a_{h + 1}$, and
		\[a_i \not \equiv a_j\pmod {2}\]
		for some  $i, j \in [1, h + 1]$, where $i \neq j$. Let $A_i = A \setminus \{a_i\}$, and let $A_j = A \setminus \{a_j\}$. Then the sumsets $h_{\pm}^\wedge A_i$ and $h_{\pm}^\wedge A_j$ are disjoint. This is proved as follows.  Let $x \in h_{\pm}^\wedge A_i \cap h_{\pm}^\wedge A_j$. Then $x = \epsilon_1 a_i + x_1 = x_2 + \epsilon_2 a_j$ for some $x_1 \in (h - 1)_{\pm}^\wedge B$ and $x_2 \in (h - 1)_{\pm}^\wedge B$, where $B = A_i \cap A_j$ and $\epsilon_1, \epsilon_2 \in \{- 1, 1\}$. Since $x_1 \equiv x_2 \pmod 2$, it follows that 
		\[a_i \equiv a_j \pmod 2,\]
		which is a contradiction. Therefore, the sumsets $h_{\pm}^\wedge A_i$ and $h_{\pm}^\wedge A_j$ are disjoint.

	\item Let $h \geq 2$ be an integer. Let $A = \{a_1,\ldots, a_h\}$ be a set of integers such that $a_1 < \cdots < a_h$. Then it is easy to show that
		\[h_{\pm}^\wedge A = \min (h_{\pm}^\wedge A) + 2 \ast \Sigma (A) = \max (h_{\pm}^\wedge A) - 2 \ast \Sigma (A).\]
\end{enumerate}

	\begin{lemma}\label{rss-lem1}
		Let $h \geq 3$ be an integer. Let $A = \{a_1,\ldots, a_{h + 1}\}$ be a set of positive integers such that $a_1 < \cdots < a_{h + 1}$. Furthermore, assume that
		\[a_1 \equiv a_2\pmod {2} ~\text{and}~ a_r \not \equiv a_1 \pmod {2}\]
		for some  $r \in [3, h + 1]$. Then
		\[|h_{\pm}^\wedge A| \geq |h_{\pm}^\wedge A_r| + \frac{h (h + 1)}{2} + 2h + 1,\]
		where $A_r = A \setminus \{a_r\}$. Hence 
		\[|h_{\pm}^\wedge A| \geq h^2 + 3h + 2.\]	
	\end{lemma}
	
	\begin{proof}  
		Let $A_1 = A \setminus \{a_1\}$. Then
		\[h_{\pm}^\wedge A_r \cup h_{\pm}^\wedge A_1 \subseteq h_{\pm}^\wedge A.\]
		Let $u = - a_2 - \cdots - a_{h + 1}$. Define
		\begin{align*}
			B_0 &= \{u\},\\
			B_1 &= u + 2 \ast \{a_i : i = 2, \ldots, h + 1\}.
		\end{align*}
		Furthermore, for $j = 2, \ldots, h$, define
		\[B_j = u + 2 \ast \{a_{h - j + 3} + \cdots + a_{h + 1} + a_i : i = 2, \ldots, h - j + 2\}. \]
		Since 
		\[\max (B_i) < \min (B_{i + 1})\]
		for $i = 0, \ldots, h - 1$, it follows that the sets $B_i$ are pairwise disjoint. Since $B_j \subseteq  h_{\pm}^\wedge A_1$ for $j = 0, \ldots, h$, it follows that 
		\[B_0 \cup \cdots \cup B_h \subseteq h_{\pm}^\wedge A_1 \subseteq h_{\pm}^\wedge A.\]
		Since the sumsets $h_{\pm}^\wedge A_1$ and $h_{\pm}^\wedge A_r$ are disjoint, it follows that $B_0 \cup \cdots \cup B_h$ and $h_{\pm}^\wedge A_r$ are disjoint sets. Hence
		\[h_{\pm}^\wedge A \supseteq h_{\pm}^\wedge A_r \cup B_0 \cup \cdots \cup B_h,\]
		and so
		\begin{align*}
			|h_{\pm}^\wedge A| &\geq |h_{\pm}^\wedge A_r| + \sum_{j = 0}^{h} |B_j|\\
			&= |h_{\pm}^\wedge A_r| + 1 + \sum_{j = 1}^{h}(h - j + 1)\\
			&= |h_{\pm}^\wedge A_r| + \frac{h(h + 1)}{2} + 1.
		\end{align*}
		To prove the lemma, it suffices to construct a set of $2h$ more elements of $h_{\pm}^\wedge A$ distinct from the elements of $h_{\pm}^\wedge A_r \cup (B_0 \cup \cdots \cup B_h)$. Let $v = - a_3 - \cdots - a_{h + 1}$. Define
		\[C_1 = \{a_1 + v, - a_1 + v\}.\]
		For $j = 2, \ldots, h$, define
		\[C_j = \{a_1 + v + 2(a_{h - j + 3} + \cdots + a_{h + 1 }), - a_1 + v + 2(a_{h - j + 3} + \cdots + a_{h + 1 })\}.\]
		It is easy to see that
		\[\max (B_i) < \min (C_{i + 1 }) < \max (C_{i + 1}) < \min (B_{i + 1})\]
		for $i = 0, \ldots, h - 1$. Thus all sets $B_i$ and $C_j$ are pairwise disjoint. Let
		\[S = C_1 \cup \cdots \cup C_h.\] 
		Then $S \subseteq h_{\pm}^\wedge A$. Now, we show that $S$ and $h_{\pm}^\wedge A_r$ are disjoint sets. Let $C = A \setminus \{a_2, a_r\}$. Let $x \in h_{\pm}^\wedge A_r \cap S$. Then $x = \epsilon_1 a_2 + x_1 = x_2 + \epsilon_2 a_r$ for some $x_1 \in (h - 1)_{\pm}^\wedge C$ and $x_2 \in (h - 1)_{\pm}^\wedge C$, where $\epsilon_1, \epsilon_2 \in \{- 1, 1\}$. Since $x_1 \equiv x_2 \pmod 2$, it follows that $a_r \equiv a_2 \pmod 2$ which is a contradiction. Therefore, the sumset $h_{\pm}^\wedge A_r$ and the set $S$ are disjoint. Since $|S| = 2h$, the set $S$ contains $2h$ elements of $h_{\pm}^\wedge A$ distinct from the elements of  $h_{\pm}^\wedge A_r \cup (B_0 \cup \cdots \cup B_h)$. Therefore,
		\begin{align*}
			|h_{\pm}^\wedge A| &\geq |h_{\pm}^\wedge A_r| + \sum_{j = 0}^{h} |B_j| + |S|\\
			&= |h_{\pm}^\wedge A_r| + \frac{h(h + 1)}{2} + 2h + 1.
		\end{align*}
		Now it follows from Theorem \ref{thm:3} that
		\[|h_{\pm}^\wedge A| \geq h^2 + 3h + 2.\]
		This completes the proof.
\end{proof}

\begin{lemma}\label{rss-lem2}
		Let $h \geq 3$ be an integer. Let $A = \{a_1, a_2,\ldots, a_{h}\}$ be a set of odd positive integers such that $a_1 < \cdots < a_{h}$. Then 
		\begin{equation}\label{eq:13}
			|h_{\pm}^\wedge A|\geq h^2 - 1.
		\end{equation}	
		The lower bound in \eqref{eq:13} is best possible.
\end{lemma}
	
	\begin{proof}
		Since $|h_{\pm}^\wedge A| = |\Sigma (A)|$, it is enough to show  that 
		$|\Sigma (A)| \geq h^2 - 1$. In case of $h = 3$, we note that $a_3 \neq a_1 + a_2$ because $a_1$, $a_2$, $a_3$ are odd. Hence
		\[\{0, a_1, a_2, a_3, a_1 + a_2, a_1 + a_3, a_2 + a_3 , a_1 + a_2 + a_3\} = \Sigma (A),\]
		and so
		\[|\Sigma (A)| = 8 = 3^2 - 1.\]
		Now assume that $h \geq 4$. Define the subsets $B_0, B_1, \ldots, B_h, C_1, C_2, \ldots, C_{h - 2}$ of $\Sigma (A)$ as follows: 
		\begin{align*}
			B_0 & = \{0\},\\
			B_1 & = \{a_i : i = 1, \ldots, h\},\\
			C_1 & = a_1 + \{a_i : i = 2, \ldots, h - 1\}.
		\end{align*}
		For $j = 2, \ldots, h$, we define
		\[B_j = \{a_i : i = 1, 2, \dots, h + 1 - j\} + a_{h - j + 2} + \cdots + a_{h}.\] 
		Furthermore, for $j = 2, \ldots, h - 2$, define
		\[C_j = a_1 + \{a_i : i = 2, \dots, h - j\} + a_{h - j + 2} + \cdots + a_{h}.\]
		Observe the following:
		\begin{enumerate} 
			\item Since 
			\[\max (B_i) < \min (B_{i + 1})\]
			for $i = 0, 1, \ldots, h$, it follows that sets  $B_0, B_1, \ldots, B_h$ are pairwise disjoint. 
			\item  Similarly all $C_j$ are disjoint for $j \in [1, h - 2]$.
			\item  For each $i \in [1, h - 2]$, by comparing the elements of $B_i$ and $C_i$ modulo $2$, we see that the sets $B_i$ and $C_i$ are disjoint.
			\item Since 
			\[\max (B_i) < \min (C_{i + 1})\]
			for each $i \in [1, h - 3]$, it follows that $B_i$ and $C_{i + 1}$ are disjoint for each $i \in [1, h - 3]$.
			\item Since 
			\[\max (C_i) < \min (B_{i + 1})\]
			for each $i = [1, h - 2]$, it follows that $C_i$ and $B_{i + 1}$ are disjoint for each $i \in [1, h - 2]$.
			\item Since 
			\[\max (C_{h - 2}) < \min (B_h),\]
			it follows that $C_{h - 2}$ and $B_h$ are disjoint.
			\item Clearly, for each $i \in [1, h - 2]$, the set $C_i$ is disjoint from $B_0$.
		\end{enumerate}
		From the above observation we see that the sets $B_i$ and $C_j$ are pairwise disjoint. Let 
		\[X = (B_0 \cup \cdots \cup B_h) \cup (C_1 \cup \cdots \cup C_{h - 2}).\]
		Since $X \subseteq \Sigma (A)$, it follows that
		\begin{align*}
			|\Sigma (A)| \geq |X| &= \sum_{j = 0}^{h} |B_j| + \sum_{j = 1}^{h - 2} |C_j|\\
			&= \sum_{j = 1}^{h}|B_j| + 1 + \sum_{j = 1}^{h - 2}|C_j|\\
			&= \sum_{j = 1}^{h}(h - j + 1) + 1 + \sum_{j = 1}^{h - 2}(h - j - 1)\\
			&= h^2 - h + 2. 
		\end{align*}
		To prove the lemma, it suffices to construct $h - 3$ more elements in $\Sigma (A)$ distinct from the elements of $X$. For each $j \in [1, h - 3]$, let 
		\[\delta_j = \max (B_j) - {\max}_{-} (B_{j}) = a_{h - j + 1} - a_{h - j},\]
and
		\begin{equation}
			\alpha_j = 
			\begin{cases}
				{\max}_{-} (B_{j}) + a_2, &~\text{if}~  \delta_j \geq a_1 + a_2;\\
				{\max}_{-} (B_{j}) + a_1 + a_2, &~\text{if}~ \delta_j < a_1 + a_2.   
			\end{cases}
		\end{equation}
Let
		\[Y = \{\alpha_j : j \in  [1, h - 3]\}.\]
Now observe the following.
		\begin{enumerate}
			\item  If $\delta_j \geq a_1 + a_2$, then
			\[{\max}_{-} (B_{j}) < \alpha_j < \max (B_j),\] 
			and thus the element $\alpha_j$ is an extra element of $\Sigma (A)$ distinct from the elements of $X$.
			\item  If $\delta_j < a_1 + a_2$, then
			\[\max (B_j) < \alpha_j < \min (C_{j + 1}),\]
			and thus the element $\alpha_j$ is an extra element of $\Sigma (A)$ distinct from the elements of $X$. 
		\end{enumerate}
		Thus for each $j \in [1, h - 3]$, we get one extra element and these elements are in $\Sigma (A)$, which are different elements from the elements of $X$. 
		Furthermore, it is easy to see that the sets $X$ and $Y$ are disjoint subsets of $\Sigma (A)$. Let $S = X \cup Y$. Then $S \subseteq \Sigma (A)$, and so
		\begin{align*}
			|\Sigma (A)| \geq |S| &= \sum_{j = 0}^{h} |B_j| + \sum_{j = 1}^{h - 2} |C_j| + |Y|\\
			&= \sum_{j = 0}^{h} |B_j| + \sum_{j = 1}^{h - 2} |C_j| + (h - 3)\\
			&= (h^2 - h + 2) + (h - 3)\\
			&=  h^2 - 1.
		\end{align*}
		Therefore,
		\begin{equation}\label{eq:14}
			|\Sigma (A)| \geq h^2 - 1. 
		\end{equation}
		Next we show that this lower bound is best possible. Let $h \geq 3$ be an integer, and let $A = \{1, 3, \ldots, 2h - 1\}$. Then 
		\[\Sigma (A) \subseteq [0, h^2],\]
		It is easy to see that 
		\[2 \notin \Sigma (A)\ \text{and}\ h^2 - 2 \notin \Sigma (A).\]
		Therefore,
		\[\Sigma (A) \subseteq [0, h^2] \setminus \{2, h^2 - 2\},\]
		and so
		\[|\Sigma (A)| \leq (h^2 + 1) - 2 = h^2 - 1.\]
		This inequality together with the inequality \eqref{eq:14} implies that
		\[|\Sigma (A)| = h^2 - 1,\]
		and so
		\[|h_{\pm}^\wedge A| = |\Sigma (A)| = h^2 - 1.\]
		Thus the lower bound in \eqref{eq:13} is best possible. This completes the proof.
	\end{proof}
	
	\begin{lemma}\label{rss-lem6}
		Let $h \geq 5$ be an integer. Let $A = \{a_1, a_2, \ldots, a_{h}\}$ be a set of odd positive integers such that $a_1 < a_2 < \cdots < a_{h}$. Then 
		\[|h_{\pm}^\wedge A| = h^2 - 1,\]
		if and only if,
		\[A = a_1 \ast \{1, 3, \ldots, 2h - 1\}.\]	
	\end{lemma}
	
	\begin{proof} First assume that $A = a_1 \ast \{1, 3, \ldots, 2h - 1\} = a_1 \ast B$, where $B = \{1, 3, \ldots, 2h - 1\}$. It has been shown in the proof of the previous lemma that $|h_{\pm}^\wedge B| = h^2 - 1$, and so
\[|h_{\pm}^\wedge A| = |a_1 \ast h_{\pm}^\wedge B| = |h_{\pm}^\wedge B| = h^2 -1.\]

Conversely, assume that $|h_{\pm}^\wedge A| = h^2 - 1$. Let $X, Y, S, B_0, \ldots, B_h,$ and $C_1, \ldots, C_{h - 2}$ be the sets as defined in the proof of Lemma \ref{rss-lem2}. Since $|h_{\pm}^\wedge A| = |\Sigma (A)|$, it follows that 
\[|\Sigma (A)| = h^2 - 1.\]
Thus $\Sigma (A)$ contains precisely the elements of the set $S$. Recall also from the proof of Lemma \ref{rss-lem2} that
		\[\delta_j = \max (B_j) - {\max}_{-} (B_{j}) = a_{h - j + 1} - a_{h - j}\]
		for each $j \in [1, h - 3]$.
  
\noindent {\textbf{Claim 1.}} $\delta_j < a_1 + a_2$.

		First we show that $\delta_j \leq a_1 + a_2$ for all $j \in [1, h - 3]$. If $\delta_j > a_ 1 + a_2$ for some $j = j_0 \in [1, h - 3]$, then
		\[a_{h - j_0} < a_{h - j_0} + a_1 + a_2 < a_{h - j_0 + 1},\]
		where $a_{h - j_0}, a_{h - j_0 + 1} \in B_1$. It is easy to verify the following.
		\begin{enumerate}
			\item Since $a_{h - j_0} + a_1 + a_2 < \min (Y)$, it follows that $a_{h-j_0} + a_1 + a_2 \not\in Y$.
			\item Since 
\[a_{h - j_0} + a_1 + a_2 \not\in B_0 \cup B_1 \cup C_1\]
and 
\[a_{h - j_0} + a_1 + a_2 < \min (B_2 \cup \cdots \cup B_h \cup C_2 \cup \cdots \cup C_{h - 2}),\]
it follows that $$a_{h - j_0} + a_1 + a_2 \not \in X.$$ Thus $$a_{h - j_0} + a_1 + a_2 \not \in S.$$
		\end{enumerate}
		Since $a_{h - j_0} + a_1 + a_2 \in \Sigma (A) \setminus S$, it follows that $|\Sigma (A)| \geq |S| + 1 = h^2$, which is a contradiction. Hence 
		\[\delta_j \leq a_1 + a_2\]
		for each $j \in [1, h - 3]$. 
		
		Now we show that $\delta_j \neq a_1 + a_2$ for each $j \in [1, h - 3]$. Suppose that $\delta_j = a_ 1 + a_2$ for some $j = j_0 \in [1, h - 3]$. We consider the following cases.

		\noindent {\textbf{Case 1}} ($j_0 \in [1, h - 4]$). In this case, $\delta_{j_0} = a_ 1 + a_2$ implies that
		\[a_{h - j_0 + 1} = a_{h - j_0} + a_1 + a_2.\]
		Consider the following inequalities:
		\[a_{h - j_0 - 1} < a_{h - j_0 - 1} + a_1 + a_2 < a_{h - j_0} + a_1 + a_2 = a_{h - j_0 + 1},\] 
		\[a_{h - j_0 - 1} < a_{h - j_0} < a_{h - j_0 + 1},\]
		Since $\Sigma(A)$ can not have two elements between the elements $a_{h - j_0 - 1} \in S$ and $a_{h - j_0 + 1} \in S$ (otherwise, $\Sigma(A)$ will have more than $h^2 - 1$ elements), it follows that
		\[a_{h - j_0} = a_{h - j_0 - 1} + a_1 + a_2,\]
		and so
		\[a_{h - j_0 - 1} + a_1 < a_{h - j_0 - 1} + a_2 = a_{h - j_0} - a_1 < a_{h - j_0} + a_1,\]
		where $a_{h - j_0 - 1} + a_1, a_{h - j_0} + a_1 \in C_1$. It is easy to verify the following.
		\begin{enumerate}
			\item Since $a_{h - j_0 - 1} + a_2 < \min (Y)$, it follows that $a_{h - j_0} + a_1 + a_2 \notin Y$.
			\item Since 
\[a_{h - j_0 - 1} + a_2 \notin B_0 \cup B_1 \cup C_1\]
 and 
 \[a_{h - j_0 - 1} + a_2 < \min (B_2 \cup \cdots \cup B_h \cup C_2 \cup \cdots \cup C_{h - 2}),\]
  it follows that $$a_{h - j_0} + a_1 + a_2 \notin X.$$ Thus $$a_{h - j_0} + a_1 + a_2 \notin S.$$
		\end{enumerate}
		Hence $a_{h - j_0 - 1} + a_2 \in \Sigma (A) \setminus S$, and so
\[|\Sigma (A)| \geq |S| + 1 = h^2,\]
which is a contradiction. Therefore, $\delta_j \neq a_1 + a_2$ for each $j \in [1, h - 4]$. Hence $\delta_j < a_1 + a_2$ which proves Claim $1$ for each $j \in [1, h - 4]$. 
		
		\noindent {\textbf{Case 2}} ($j_0 = h - 3$). In this case, $\delta_{j_0} = a_1 + a_2$ implies that
		\[a_4 = a_3 + a_2 + a_1.\]
		Since $a_4 = a_3 + a_2 + a_1$, it follows that 
		\[a_3 < a_3 + a_2 < a_4\]
		and
		\[a_3 + a_1 < a_3 + a_2 < a_4 + a_1.\]
		It is easy to see that $a_3 + a_2 \in \Sigma (A) \setminus S$. Hence
		\[|\Sigma (A)| \geq |S| + 1 = h^2,\] 
		which is a contradiction. Therefore, $\delta_{j_0} \neq a_1 + a_2$,
		and hence $\delta_{j_0} < a_1 + a_2$. 

Thus we have shown that $\delta_j < a_1 + a_2$ for each $j \in [1, h-3]$ which proves Claim $1$.
 	
		Since $\delta_j < a_1 + a_2$, it follows that 
		\[\alpha_j = {\max}_{-} (B_{j}) + a_1 + a_2~ \text{for each}~ j \in [1, h - 3],\]
		and so
		\[Y = \{{\max}_{-} (B_{j}) + a_1 + a_2 : j \in [1, h - 3]\}.\]
		
		\noindent {\textbf{Claim 2}} ($\delta_j = a_2 - a_1$ for each $j \in [1, h - 3]$).
		
		First we show that $\delta_j \geq a_2 - a_1$ for each $j \in [1, h -3]$. If $\delta_j < a_ 2 - a_1$ for some $j = j_0 \in [1, h - 3]$, then
		\[\max (B_{j_0}) < {\max}_{-} (B_{j_0}) + a_2 - a_1.\]
		It is easy to verify the following.
		\begin{enumerate}
			\item Since $${\max}_{-} (B_{j_0}) \not \equiv {\max}_{-} (B_{j_0 - 1}) \pmod 2$$ and $$\max (B_{j_0}) < {\max}_{-} (B_{j_0}) + a_2 - a_1 < \alpha_{j_0} < \min (C_{j_0 + 1}),$$ it follows that $${\max}_{-} (B_{j_0}) + a_2 - a_1 \not\in Y.$$
			\item Since 
\[{\max}_{-} (B_{j_0}) + a_2 - a_1 \not\in B_0 \cup \cdots \cup B_{j_0 + 1} 
			\cup C_1 \cup \cdots \cup C_{j_0}\]
 and 
 \[{\max}_{-} (B_{j_0}) + a_2 - a_1 < \min (B_{j_0 + 2} \cup \cdots \cup B_h \cup C_{j_0 + 1} \cup \cdots \cup C_{h - 2}),\]
  it follows that $${\max}_{-} (B_{j_0}) + a_2 - a_1 \not\in X.$$ Hence $${\max}_{-} (B_{j_0}) + a_2 - a_1 \not\in S.$$
		\end{enumerate}
		Since ${\max}_{-} (B_{j_0}) + a_2 - a_1 \in \Sigma (A) \setminus S$, it follows that $|\Sigma (A)| \geq |S| + 1 = h^2$, which is a contradiction. Hence 
		\[\delta_j \geq a_2 - a_1\]
		for each $j \in [1, h - 3]$. 
		
		Next we show that $\delta_j = a_2 - a_1$ for each $j \in [1, h -3]$. If $\delta_j < a_ 2 - a_1$ for some $j = j_0 \in [1, h - 3]$, then
		\[\max (B_{j_0}) > {\max}_{-} (B_{j_0}) + a_2 - a_1.\]
		It is easy to verify the following.
		\begin{enumerate}
			\item Since 
              \[{\max}_{-} (B_{j_0)}) \not \equiv {\max}_{-} (B_{j_0 - 1}) \pmod 2\]
               and 
              \[{\max}_{-}(B_{j_0}) < {\max}_{-} (B_{j_0}) + a_2 - a_1 < \max (B_{j_0}) < \alpha_{j_0} < \min (C_{j_0 + 1}),\]
              it follows that 
              \[{\max}_{-} (B_{j_0}) + a_2 - a_1 \not\in Y.\]
			\item  Since 
               \[{\max}_{-} (B_{j_0}) + a_2 - a_1 \not\in B_0 \cup \cdots \cup B_{j_0 + 1} \cup C_1 \cup \cdots \cup C_{j_0}\]
                and 
                \[{\max}_{-} (B_{j_0}) + a_2 - a_1 < \min (B_{j_0 + 2} \cup \cdots \cup B_h \cup C_{j_0 + 1} \cup \cdots \cup C_{h - 2}),\]
                it follows that $${\max}_{-} (B_{j_0}) + a_2 - a_1 \notin X.$$ Thus 
                \[{\max}_{-} (B_{j_0}) + a_2 - a_1 \not\in S.\]
		\end{enumerate}
		Since ${\max}_{-} (B_{j_0}) + a_2 - a_1 \in \Sigma (A) \setminus S$, it follows that $|\Sigma (A)| \geq |S| + 1 = h^2$, which is a contradiction. 
		Therefore, $\delta_j = a_2 - a_1$ for each $j \in [1, h - 3]$ which proves Claim $2$. 
 
        Thus we have
		\begin{equation}\label{eq:5}
			a_2 - a_1 = a_4 - a_3 = \cdots = a_h - a_{h - 1}.
		\end{equation}

		Next we show that $a_3 - a_2 = a_2 -  a_1$. Consider the following elements of $S$ between $a_{h - 1} + a_1 + a_2$ and $ a_h + a_1 + a_{h - 2}$ in increasing order.
		\[a_{h - 1} + a_1 + a_2 < a_{h - 1} + a_1 + a_3 < \cdots < a_{h - 1} + a_1 + a_{h - 2} < a_h + a_1 + a_{h - 2}.\]
		We have the following inequality also.
		\[a_{h - 1} + a_1 + a_2 < a_h + a_1 + a_2 < \cdots < a_h + a_1 + a_{h - 3} < a_h + a_1 + a_{h - 2}.\]
		Since $|h_{\pm}^\wedge A| = h^2 - 1$, it follows that
		\[a_{h - 1} + a_1 + a_{j + 1} = a_h + a_1 + a_j\]
		for each $j \in [2, h - 3]$. Thus 
		\begin{equation}\label{eq:6}
			a_h - a_{h - 1} = a_3 - a_2 = \cdots = a_{h - 2} - a_{h - 3}.
		\end{equation}
		Therefore, it follows from \eqref{eq:5} and \eqref{eq:6} that
		\begin{equation}\label{eq:7}
			a_2 - a_1 = a_3 - a_2 = \cdots = a_{h - 1} - a_{h - 2} = a_h - a_{h - 1}.
		\end{equation}

		Now we show that $a_h = a_{h - 2} + a_1 + a_2$. Clearly,
		\[a_{h - 1} = a_{h - 2} + a_2 - a_1 < a_{h - 2} + a_1 + a_2 < \alpha_1 <a_h + a_1 + a_2\]
		and
		\[a_{h - 1} < a_h < \alpha_1 <a_h + a_1 + a_2.\]
		Since $|h_{\pm}^\wedge A| = h^2 - 1$, it follows that 
		\begin{equation}\label{eq:8}
			a_h = a_{h - 2} + a_1 + a_2.
		\end{equation}
		Using \eqref{eq:7} and \eqref{eq:8}, we have
		\begin{align*}
			a_1 + a_2 &= a_{h} - a_{h -2}\\
			&= (a_{h} - a_{h -1}) + (a_{h - 1} - a_{h -2})\\
			&= (a_2 - a_1) + (a_2 - a_1)\\
			&= 2a_2 - 2a_1.
		\end{align*}
		Therefore, 
		\begin{equation}\label{eq:9}
			a_2 - a_1 = 2a_1.
		\end{equation} 
		Hence it follows from \eqref{eq:7} and \eqref{eq:9} that
		\[a_2 = 3a_1, a_3 = 5a_1, \ldots, a_{h - 1} = (2h - 3)a_1, a_h = (2h - 1)a_1.\]
		Therefore,
		\[A = a_1 \ast \{1, 3, \ldots, 2h - 1\}.\]	
		This completes the proof.
	\end{proof}
	
	\begin{lemma}\label{rss-lem19}
		Let $A = \{a_1, a_2, a_3, a_4\}$ be a set of odd positive integers such that $a_1 < a_2 < a_3 < a_4$. Then 
		\[|4_{\pm}^\wedge A| = 15,\]
		if and only if either
		\[A = \{a_1, a_2, a_3, a_3 + a_2 + a_1\}\]
		or
		\[A = \{a_1, a_2, a_3, a_3 + a_2 - a_1\}.\]
	\end{lemma}
	
	\begin{proof}
If $A = \{a_1, a_2, a_3, a_4\}$, where $a_4 = a_3 + a_2 + a_1$, then
\begin{align*}
  \Sigma(A) = \{0, a_1, a_2, a_3, a_4, & \ a_1 + a_2, a_1 + a_3, a_2 + a_3, a_1 + a_4, a_2 + a_4, a_3 + a_4, \\
  & a_1 + a_2 + a_4, a_1 + a_3 + a_4, a_2 + a_3 + a_4, a_1 + a_2 + a_3 + a_4\},
\end{align*}
and so
\[|4_{\pm}^\wedge A| = |\Sigma(A)| = 15.\]
Similarly, if $A = \{a_1, a_2, a_3, a_4\}$, where $a_4 = a_3 + a_2 - a_1$, then
\begin{align*}
  \Sigma(A) = \{0, a_1, a_2, a_3, a_4, & \ a_1 + a_2, a_1 + a_3, a_1 + a_4, a_2 + a_4, a_3 + a_4, a_1 + a_2 + a_3, \\
  & a_1 + a_2 + a_4, a_1 + a_3 + a_4, a_2 + a_3 + a_4, a_1 + a_2 + a_3 + a_4\},
\end{align*}
and so
\[|4_{\pm}^\wedge A| = |\Sigma(A)| = 15.\]
		
Conversely, assume that $|4_{\pm}^\wedge A|= 15$. Then $$|\Sigma (A)| = |4_{\pm}^\wedge A|= 15,$$ and so it follows from Lemma \ref{rss-lem2} that $\Sigma (A)$ contains precisely the elements of the set 
\[S = B_0 \cup B_1 \cup C_1 \cup B_2 \cup C_2 \cup B_3 \cup B_4 \cup Y\]
which was constructed in the proof of Lemma \ref{rss-lem2}. The sets in the union are precisely the following sets. 
\begin{align*}
  B_0 & = \{0\}, \\
  B_1 & = \{a_1, a_2, a_3, a_4\}, \\
  C_1 & = \{a_1 + a_2, a_1 + a_3\}, \\
  B_2 & = \{a_1 + a_4, a_2 + a_4, a_3 + a_4\}, \\
  C_2 & = \{a_1 + a_2 + a_4\}, \\
  B_3 & = \{a_1 + a_3 + a_4, a_2 + a_3 + a_4\},\\
  B_4 & = \{a_1 + a_2 + a_3 + a_4\},\\
  Y & = \{\alpha_1\},
\end{align*}  
where $\alpha_1$ is defined as follows. 
\begin{equation*}
  \alpha_1 =
  \begin{cases}
    a_3 + a_2, & \mbox{if } \delta_1 \geq a_1 + a_2; \\
    a_3 + a_2 + a_1, & \mbox{if } \delta_1 < a_1 + a_2,
  \end{cases}
\end{equation*}	
where $\delta_1 = a_4 - a_3$.
		
\noindent {\textbf{Claim.}} Either $\delta_1 = a_1 + a_2$ or $\delta_1 = a_2 - a_1$. 
		
If $\delta_1 > a_1 + a_2$, then $\alpha_1 =  a_3 + a_2$, and
		\[a_3 < a_3 + a_2 + a_1 < a_4,\]
and so
\[a_3 + a_2 + a_1 \in \Sigma (A) \setminus S.\]
Hence it follows that $|\Sigma (A)| \geq |S| + 1 = 16$, which is a contradiction. 

Now assume that $\delta_1 < a_2 - a_1$. In this case, $\alpha_1 =  a_3 + a_2 + a_1$, and we have
		\[a_4 + a_1 < a_3 + a_2 < a_4 + a_2,\]
		and
		\[a_4 < a_3 + a_2 + a_1 < a_4 + a_2 + a_1.\]
		Since $a_3 + a_2 \in \Sigma (A) \setminus S$, it follows that $|\Sigma (A)| \geq |S| + 1 = 16$, which is a contradiction.

Finally, assume that $a_2 - a_1 < \delta_1 < a_1 + a_2$. In this case, $\alpha_1 =  a_3 + a_2 + a_1$, and we have
		\[a_3 + a_1 < a_3 + a_2 < a_4 + a_1,\]
		and
		\[a_4 < a_3 + a_2 + a_1 < a_4 + a_2 + a_1.\]
		Since $a_3 + a_2 \in \Sigma (A) \setminus S$, it follows that $|\Sigma (A)| \geq |S| + 1 = 16$, which is again a contradiction. 
Therefore, the only possibility is that either $\delta_1 = a_2 - a_1$ or $\delta_1 = a_1 + a_2$. Hence either
		\[A = \{a_1, a_2, a_3, a_3 + a_2 + a_1\}\]
or
\[A = \{a_1, a_2, a_3, a_3 + a_2 - a_1\}.\]
This completes the proof.
	\end{proof}
	
	\begin{lemma}\label{rss-lem4}
		Let $h \geq 3$ be an integer. Let $A = \{a_1, \ldots, a_{h + 1}\}$ be a set of positive integers such that $a_1 < \cdots < a_{h + 1}$. Furthermore, assume that 
		\[a_1 \not \equiv a_2 \pmod {2} ~\text{and}~ a_1 \not \equiv a_3 \pmod {2}.\]
		Then
		\begin{equation*}
			|h_{\pm}^\wedge A| \geq 
           \begin{cases}
              |h_{\pm}^\wedge A_1| + \frac{h(h + 1)}{2} + 2h - 1, & \mbox{if } a_3 = 2a_1 + a_2; \\
              |h_{\pm}^\wedge A_1| + \frac{h(h + 1)}{2} + 3h - 2, & \mbox{if } a_3 \neq 2a_1 + a_2,
           \end{cases}	
		\end{equation*}
		where $A_1 = A \setminus \{a_1\}$. Hence 
		\begin{equation*}
			|h_{\pm}^\wedge A| \geq 
			\begin{cases}
				h^2 + 3h, & \mbox{if }~  a_3 = 2a_1 + a_2;\\
				h^2 + 4h - 1, & \mbox{if }~ a_3 \neq 2a_1 + a_2.   
			\end{cases}
		\end{equation*}
	\end{lemma} 
	
	\begin{proof}
		Let $A_2 = A \setminus \{a_2\}$. Then
		\[h_{\pm}^\wedge A_1 \cup h_{\pm}^\wedge A_2 \subseteq h_{\pm}^\wedge A.\]
		Let $u = - a_1 - a_3 - \cdots - a_{h + 1}$, and let $v = - a_1 - a_2 - a_4 - \cdots - a_{h + 1}$. Define the subsets $B_0, \ldots, B_h, C_0, \ldots, C_{h - 2}$ of $h_{\pm}^\wedge A_2$ as follows.
		\begin{align*}
			B_0 & = \{u\},\\
			B_1 & = \{u + 2a_i : i = 3, \ldots, h + 1\},\\
			C_0 & = \{u + 2a_1\} \cup \{ v, v + 2a_1, v + 2a_2\},\\
			C_1 & = 2a_{h + 1} + C_0,\\
			B_{h - 1} & = \{- u - 2a_1\},\\	
			B_h & = \{- u\}.
		\end{align*}	
		Furthermore, for each $j \in [2, h - 2]$, define
		\[B_j = \{u + 2a_i + a_{h + 3 - j} + \cdots + a_{h + 1} : i = 3, \dots, h + 2 - j\},\] 
and
		\[C_j = 2(a_{h + 2 - j} + \cdots + a_{h + 1}) + C_0.\]
		Observe the following:
		\begin{enumerate} 
			\item If $a_3 = 2a_1 + a_2$, then $v = u + 2a_1$. Hence $|C_0| = 3$. Therefore,
\[|C_j| = 3~ \text{for}~j = 0, \ldots, h - 2.\]
			\item If $a_3 \neq 2a_1 + a_2$, then $v \neq u + 2a_1$. Hence $|C_0| = 4$. Therefore,
\[|C_j| = 4~ \text{for}~j = 0, \ldots, h - 2.\]
			\item It is easy to see that 	
			\[\max (B_i) < \min (C_i) < \max (C_i) < \min (B_{i + 1})\]
			for $i = 0, 1, \ldots, h - 2$, and
			\[\max (B_{h - 1}) < \min (B_h).\]
		\end{enumerate}
		From the above observations, it follows that the sets $B_i$ and $C_j$ are pairwise disjoint for $i = 0, 1, \ldots, h$ and $j = 0, 1, \ldots, h-2$. Since the sumsets $h_{\pm}^\wedge A_1$ and $h_{\pm}^\wedge A_2$ are disjoint subsets of $h_{\pm}^\wedge A$, it follows that $B_0 \cup \cdots \cup B_h \cup C_0 \cup \cdots \cup C_{h - 2}$ and $h_{\pm}^\wedge A_1$ are disjoint subsets of $h_{\pm}^\wedge A$. Hence
\[h_{\pm}^\wedge A \supseteq h_{\pm}^\wedge A_1 \cup \supseteq h_{\pm}^\wedge A_2 \supseteq h_{\pm}^\wedge A_1 \cup B_0 \cup \cdots \cup B_h \cup C_0 \cup \cdots \cup C_{h - 2},\]
		and so
		\begin{align*}
			|h_{\pm}^\wedge A| &\geq |h_{\pm}^\wedge A_1| + \sum_{j = 0}^{h} |B_j| + \sum_{j = 0}^{h - 2} |C_j|\\
			&= |h_{\pm}^\wedge A_1| + 2 + \sum_{j = 1}^{h - 1}|B_j| + \sum_{j = 0}^{h - 2} |C_j| \\
			&= |h_{\pm}^\wedge A_1| + 2 + \sum_{j = 1}^{h - 1}(h - j) + \sum_{j = 0}^{h - 2} |C_j|\\
			&= |h_{\pm}^\wedge A_1| + \frac{h(h - 1)}{2} + \sum_{j = 0}^{h - 2} |C_j| + 2.
		\end{align*}
		Now substituting the values of $|C_j|$, we get
		\begin{equation*}
			|h_{\pm}^\wedge A| \geq 
			\begin{cases}
				|h_{\pm}^\wedge A_1| + \frac{h(h + 1)}{2} + 2h - 1, &~\text{if}~ a_3 = 2a_1 + a_2;\\
				|h_{\pm}^\wedge A_1| + \frac{h(h + 1)}{2} + 3h - 2, &~\text{if}~ a_3 \neq 2a_1 + a_2.   
			\end{cases}
		\end{equation*}
		Therefore, an application of Theorem \ref{thm:3} gives
		\begin{equation*}
			|h_{\pm}^\wedge A| \geq 
			\begin{cases}
				h^2 + 3h, &~\mbox{if } a_3 = 2a_1 + a_2;\\
				h^2 + 4h - 1, &~\mbox{if } a_3 \neq 2a_1 + a_2.   
			\end{cases}
		\end{equation*}
		This completes the proof.
	\end{proof}
	
	\begin{lemma}\label{rss-lem5}
		Let $h \geq 4$ be an integer. Let $A = \{a_1, \ldots, a_{h + 1}\}$ be a set of positive integers such that $a_1 < \cdots < a_{h + 1}$. Furthermore, assume that 
		\[a_2 \not \equiv a_1 \pmod {2} ~\text{and}~ a_3 \equiv a_1 \pmod {2}.\]
		Then
		\[|h_{\pm}^\wedge A| \geq |h_{\pm}^\wedge A_2| + \frac{h(h + 1)}{2} + h,\]
		where $A_2 = A \setminus \{a_2\}$. Hence
		\begin{equation}\label{rss-lem5eq1}
			|h_{\pm}^\wedge A| \geq 
			\begin{cases}
				h^2 + 2h + 2, &\mbox{if } h \geq 4 ~\text{and}~ A_2 ~\text{is not an A.P.};\\
				\frac{1}{2} h(3h - 1) + 4, &\mbox{if }h \geq 4, A_2 ~\text{is an A.P. and}~ a_2 \not \equiv 0 \pmod 2;\\
				26, &\mbox{if } h = 4 ~\text{and}~ A_2 ~\text{is an A.P. and}~ a_2 \equiv 0 \pmod 2;\\
				2h(h - 1), &\mbox{if } h \geq 5 ~\text{and}~ A_2 ~\text{is an A.P. and}~ a_2 \equiv 0 \pmod 2.  
			\end{cases}
		\end{equation}
	\end{lemma} 
	
	\begin{proof}
		Let $A_1 = A \setminus \{a_1\}$ and $A_3 = A \setminus \{a_3\}$. Then
		\[h_{\pm}^\wedge A_1 \cup h_{\pm}^\wedge A_2 \cup h_{\pm}^\wedge A_3 \subseteq h_{\pm}^\wedge A.\]
		Now we define the subsets $B_0, \ldots, B_h$ of $h_{\pm}^\wedge A_1$ as follows. Let 
\[u = \min(h_{\pm}^\wedge A_1) = - a_2 - \cdots - a_{h + 1},\] 
and define
		\begin{align*}
			B_0 & = \{u\},\\
			B_1 & = \{u + 2a_i : i = 3, \ldots, h + 1\},\\
			B_{h - 1} & = \{- u - 2a_2\},\\	
			B_h & = \{- u\}.
		\end{align*}
		Furthermore, for each $j \in [2, h - 2]$, define
		\[B_j =  \{u + 2(a_i + a_{h + 3 - j} + \cdots + a_{h + 1}) : i = 3, \ldots, h + 2 - j\}.\] 
		Now we define the subsets $C_0, \ldots, C_{h - 2}$ of $h_{\pm}^\wedge A_3$ as follows. Let 
\[v = \min(h_{\pm}^\wedge A_3)= - a_1 - a_2 - a_4 - \cdots - a_{h + 1},\]
 and define
		\[C_0 = \{v, v + 2a_1\}.\]
		Furthermore, for each $j \in [2, h - 2]$, define
		\[C_j = 2(a_{h + 2 - j} + \cdots + a_{h + 1}) + C_0.\]
		It is easy to see that 	
		\[\max (B_i) < \min (C_i) < \max (C_i) < \min (B_{i + 1})\]
		for $i = 0, \ldots, h - 2$, and
		\[\max (B_{h - 1}) < \min (B_h).\]
		Hence the sets $B_i$ and $C_j$ are disjoint sets for $i = 0, \ldots, h$ and $j = 0, \ldots, h - 2$. Since the sumsets $h_{\pm}^\wedge A_2$ is disjoint with each of the sumsets $h_{\pm}^\wedge A_1$ and $h_{\pm}^\wedge A_3$, it follows that $B_0 \cup \cdots \cup B_h \cup C_0 \cdots \cup C_{h - 2}$ and $h_{\pm}^\wedge A_2$ are disjoint sets. Hence
		\[h_{\pm}^\wedge A \supseteq h_{\pm}^\wedge A_2 \cup B_0 \cup \cdots \cup B_h \cup C_0 \cup \cdots \cup C_{h - 2},\]
		and so
		\begin{align*}
			|h_{\pm}^\wedge A| & \geq |h_{\pm}^\wedge A_2| + \sum_{j = 0}^{h} |B_j| + \sum_{j = 0}^{h - 2} |C_j|\\
			&= |h_{\pm}^\wedge A_2| + 2 + \sum_{j = 1}^{h - 1}|B_j| + \sum_{j = 0}^{h - 2} 2\\
			&= |h_{\pm}^\wedge A_2| + 2 + \sum_{j = 1}^{h - 1}(h - j) + 2(h - 1)\\
			&= |h_{\pm}^\wedge A_2| + \frac{h(h + 1)}{2} + h.
		\end{align*}
		Therefore, it follows from Theorem \ref{thm:3} and Theorem \ref{thm:4} that if $A_2$ is not an arithmetic progression, then  
		\begin{equation*}
			|h_{\pm}^\wedge A| \geq h^2 + 2h + 2.
		\end{equation*}
This establishes the first inequality in $\eqref{rss-lem5eq1}$.

		Now assume that the set $A_2$ is an arithmetic progression. Since $a_3 \equiv a_1 \pmod 2$ and $A_2$ is an arithmetic progression, it follows that
		\[a_1 \equiv a_3 \equiv \cdots \equiv a_{h + 1} \pmod 2.\]
		We consider the following cases.

		\noindent {\textbf{Case 1}} ($a_2 \not \equiv 0 \pmod 2$). In this case, 
		\[a_1 \equiv a_3 \equiv \cdots \equiv a_{h + 1} \equiv 0 \pmod 2.\]
		We define the subsets $B_0, \ldots, B_{h - 1}, C_1, \ldots, C_{h - 1}$ of $\Sigma (A_1)$ as follows. Let
		\begin{align*}
			B_0 & = \{0, a_2\},\\
			B_1 & = \{a_i : i = 3, \ldots, h + 1\},\\
			C_1 & = \{a_2 + a_i : i = 3, \ldots, h + 1\}.
		\end{align*}
		Furthermore, for $j = 2, \ldots, h - 1$, we define
		\[B_j = \{a_i : i = 3 \dots, h + 2 - j\} + a_{h + 3 - j} + \cdots + a_{h + 1}.\] 
		Also, for  $j = 2, \ldots, h - 1$, we define
		\[C_j = \{a_2 + a_i : i = 3 \dots, h + 2 - j\} + a_{h + 3 - j} + \cdots + a_{h + 1}.\]
		Observe the following.
		\begin{enumerate} 
			\item Since 
			\[\max (B_i) < \min (B_{i + 1})\]
			for $i = 0, \ldots, h - 1$, it follows that the sets  $B_0, \ldots, B_{h - 1}$ all are pairwise disjoint. 
			\item  Since 
			\[\max (C_i) < \min (C_{i + 1})\]
			for $i = 1, \ldots, h - 1$, it follows that sets  $C_1, \ldots, C_{h - 1}$ all are pairwise disjoint.
           \item  Since 
			\[\max(B_0)) < \min(C_i)\]
			for each $i \in [1, h-1]$, it follows that set $B_0$ is disjoint with each of the sets $C_1, \ldots, C_{h - 1}$.  
			\item For each $i \in [1, h - 1]$, all the elements of $B_i$ are even. For each $i \in [1, h - 1]$, all the elements of $C_i$ are odd. Hence for each $i \in [1, h-1]$ and each $j \in [1, h-1]$, the sets $B_i$ and $C_j$ are disjoint sets.
			
		\end{enumerate}
		From the above observations, it follows that the sets $B_i$ and $C_j$ are pairwise disjoint for $i = 0, \ldots, h - 1$ and $j = 1, \ldots, h - 1$.
		Since 
		\[\Sigma (A_1) \supseteq B_0 \cup \cdots \cup B_{h - 1} \cup C_1 \cup \cdots \cup C_{h - 1},\]
		it follows that
		\begin{align*}
			|\Sigma (A_1)| &\geq \sum_{j = 0}^{h - 1} |B_j| + \sum_{j = 1}^{h - 1} |C_j|\\
			&= 2 + \sum_{j = 1}^{h - 1}|B_j| + \sum_{j = 1}^{h - 1}|C_j|\\
			&= \sum_{j = 1}^{h - 1}(h - j) + \sum_{j = 1}^{h - 1}(h - j) + 2\\
			&= \frac{h(h - 1)}{2} + \frac{h(h - 1)}{2} + 2\\
			&= h^2 - h + 2. 
		\end{align*}
		Since the sumsets $h_{\pm}^\wedge A_1$ and $h_{\pm}^\wedge A_2$ are disjoint, it follows that
		\begin{align*}
			|h_{\pm}^\wedge A| &\geq |h_{\pm}^\wedge A_1| + |h_{\pm}^\wedge A_2|\\
			&= |\Sigma (A_1)| + |h_{\pm}^\wedge A_2|\\
			&\geq |h_{\pm}^\wedge A_2| + h^2 - h + 2.
		\end{align*}
		
		Now we consider the following subcases of Case $1$.
		
		\noindent {\textbf{Subcase 1.1}} ($A_2 \neq a_1 \ast [1, h]$). By applying Theorem \ref{thm:3} and Theorem \ref{thm:4}, we get 
		\begin{align*}
			|h_{\pm}^\wedge A| &\geq |h_{\pm}^\wedge A_2| + h^2 - h + 2\\
			&\geq \frac{h(h + 1)}{2} + 2 + (h^2 - h + 2)\\
			&= \frac{1}{2} h(3h - 1) + 4.
		\end{align*}
		
		\noindent {\textbf{Subcase 1.2}} ($A_2 = a_1 \ast [1, h]$). Let $v = - a_1 - a_2 - a_4 - \cdots - a_{h + 1}$. Then
		\[v \not \equiv x \pmod2~ \text{for all}~x \in h_{\pm}^\wedge A_2,\]
		and so
		\[v \notin  h_{\pm}^\wedge A_2.\]
		Clearly, the first three smallest elements of $h_{\pm}^\wedge A_1$ are $- a_2 - a_3 - \cdots - a_{h + 1}$, $a_2 - a_3 - \cdots - a_{h + 1}$, $- a_2 + a_3 - \cdots - a_{h + 1}$, respectively. It is easy to see that
		\[- a_2 - a_3 - \cdots - a_{h + 1} < v < - a_2 + a_3 - \cdots - a_{h + 1}.\]
		Since $a_3 = 2a_1$, it follows that
		\[v \neq a_2 - a_3 - \cdots - a_{h + 1},\]
		and so
		\[v \notin h_{\pm}^\wedge A_1.\]
		Thus
       \[v \in h_{\pm}^\wedge A \setminus (h_{\pm}^\wedge A_1 \cup h_{\pm}^\wedge A_2),\]
       and so
		\begin{align*}
			|h_{\pm}^\wedge A| &\geq |h_{\pm}^\wedge A_1 \cup h_{\pm}^\wedge A_2| + 1\\
			&\geq |h_{\pm}^\wedge A_1| + |h_{\pm}^\wedge A_2| + 1\\
			&\geq \biggl(\frac{h(h + 1)}{2} + 1 \biggl) + (h^2 - h + 2) + 1\\
			&= \frac{1}{2} h(3h - 1) + 4.
		\end{align*}	
		Thus we have established the second inequality in $\eqref{rss-lem5eq1}$.

		\noindent {\textbf{Case 2}} ($a_2 \equiv 0 \pmod 2$).
		In this case,
		\[a_1 \equiv a_3 \equiv \cdots \equiv a_{h + 1} \not \equiv 0\pmod 2.\]
		Clearly, 
		\begin{equation}\label{eq:1}
			|h_{\pm}^\wedge A| \geq |h_{\pm}^\wedge A_1| + |h_{\pm}^\wedge A_2|.
		\end{equation}

		First assume that $h = 4$. Then by applying Lemma \ref{rss-lem2} and Theorem \ref{thm:3}, we get
		\[|h_{\pm}^\wedge A| \geq (16 - 1) + (10 + 1) = 26.\]
		This establishes the third inequality in $\eqref{rss-lem5eq1}$.

		Now assume that $h \geq 5$. Let $B = A \setminus \{a_1, a_2\}$. Clearly, 
		\[\{a_2 + \cdots + a_{h + 1}\} \cup (- a_2 + (h - 1)_{\pm}^\wedge B) \subseteq h_{\pm}^\wedge A_1,\]
		and
		\[\max (- a_2 + (h - 1)_{\pm}^\wedge B) = - a_2 + a_3 + \cdots + a_{h + 1} < a_2 + \cdots + a_{h + 1}.\]
		Hence
		\[|h_{\pm}^\wedge A_1| \geq |- a_2 + (h - 1)_{\pm}^\wedge B| + 1.\]
		By applying Lemma \ref{rss-lem2}, we get
		\begin{equation}\label{eq:2}
			|h_{\pm}^\wedge A_1| \geq 1 + (h - 1)^2 - 1 = (h - 1)^2.
		\end{equation}
		Therefore, it follows from \eqref{eq:1}, \eqref{eq:2} and Lemma \ref{rss-lem2} that
		\begin{align*}
			|h_{\pm}^\wedge A| \geq (h - 1)^2 + (h^2 - 1) = 2h(h - 1), 
		\end{align*}
which establishes the last inequality in $\eqref{rss-lem5eq1}$. This completes the proof.
	\end{proof}
	
	\begin{lemma}\label{rss-lem3}
		Let $h \geq 3$ be an integer. Let $A = \{a_1, \ldots, a_{h + 1}\}$ be a set of odd positive integers such that $a_1 < \cdots < a_{h + 1}$. Then
		\[|h_{\pm}^\wedge A| \geq |h_{\pm}^\wedge A_{h + 1}| + 2h + 2,\]
		where $A_{h + 1} = A \setminus \{a_{h + 1}\} $. Hence 
		\[|h_{\pm}^\wedge A| \geq h^2 + 2h + 1.\]	
	\end{lemma}
	
	\begin{proof}
		It is easy to see that 
		\[h_{\pm}^\wedge A_{h + 1} \cap h^\wedge A = \{a_1 + \cdots + a_h\}\]
        and
        \[h_{\pm}^\wedge A_{h + 1} \cap h^\wedge (- A) = \{- (a_1 + \cdots + a_h)\}.\]
	    Let $$C = h_{\pm}^\wedge A_{h + 1} \cup h^\wedge A \cup h^\wedge (- A).$$ Then $$ C \subseteq h_{\pm}^\wedge A,$$ and so
		\[|h_{\pm}^\wedge A| \geq |C| = |h_{\pm}^\wedge A_{h + 1}| + |h^\wedge A| + |h^\wedge (- A)| - 2.\]
		Hence by applying Theorem \ref{restricted-hfold-direct-thm}, we get 
		\begin{align*}
			|h_{\pm}^\wedge A| \geq |C| \geq |h_{\pm}^\wedge A_{h + 1}| + (h + 1) + (h + 1) - 2 = |h_{\pm}^\wedge A_{h + 1}| + 2h.
		\end{align*}
		Therefore, to prove the lemma, it suffices to construct $2$ more elements in $h_{\pm}^\wedge A$ distinct from the elements of $C$. Let 
		\begin{align*}
            z &= \max (h_{\pm}^\wedge A_{h + 1}),\\
			x &= \max (h_{\pm}^\wedge A_{h + 1}) - 2a_2 = z - 2a_2,\\
			y &= \max (h_{\pm}^\wedge A_{h + 1}) - 2a_1 = z - 2a_1,\\
			\alpha &= \max (h_{\pm}^\wedge A_{h + 1}) + a_{h + 1} - a_h - 2a_2 = z + a_{h + 1} - a_h - 2a_2,\\
			\beta &= \max (h_{\pm}^\wedge A_{h + 1}) + a_{h + 1} - a_h - 2a_1 = z + a_{h + 1} - a_h - 2a_1.
		\end{align*}
		Then $x < y < z, 0 < \alpha < \beta, x < \alpha, y < \beta$, and $\beta < {\min}_{+} (h^\wedge A)$.
		
		\noindent {\textbf{Claim 1.}} Either $\alpha \notin h_{\pm}^\wedge A_{h + 1}$ or $\beta \notin h_{\pm}^\wedge A_{h + 1}$.

		Suppose that 
		\[\alpha \in h_{\pm}^\wedge A_{h + 1}  ~\text{and}~ \beta \in h_{\pm}^\wedge A_{h + 1}.\]
		Since the first three smallest elements of $\Sigma (A_{h + 1})$ are $0, a_1, a_2$, respectively, and since
		\[h_{\pm}^\wedge A_{h + 1} = \max (h_{\pm}^\wedge A_{h + 1}) - 2 \ast \Sigma (A_{h + 1}),\] 
		it follows that the first three largest elements of $h_{\pm}^\wedge A_{h + 1}$ are $x, y, z$, respectively. Since $x < \alpha$ and $y < \beta$, it follows that 
		\[y = \alpha ~\text{and}~ z = \beta,\]
		and so
		\[a_{h + 1} - a_h = 2(a_2 - a_1) ~\text{and}~ a_{h + 1} - a_h = 2a_1.\]
		This implies that $$a_2 = 2a_1,$$ which is a contradiction. Therefore, either $\alpha \notin h_{\pm}^\wedge A_{h + 1}$ or $\beta \notin h_{\pm}^\wedge A_{h + 1}$ which proves Claim $1$.
		
		\noindent {\textbf{Claim 2.}} Either $\alpha \not\in C$ or $\beta \not\in C$.
	
		Since $\alpha $ and $ \beta$ are positive integers, it is enough to show that either $$\alpha \notin h_{\pm}^\wedge A_{h + 1} \cup h^\wedge A$$ or $$\beta \notin h_{\pm}^\wedge A_{h + 1} \cup h^\wedge A.$$ Suppose that $\alpha \in h_{\pm}^\wedge A_{h + 1} \cup h^\wedge A$  and $\beta \in h_{\pm}^\wedge A_{h + 1} \cup h^\wedge A$. Since	
		\[y = {\max}_{-} (h_{\pm}^\wedge A_{h + 1}) < \beta \leq \min (h^\wedge A) < {\min}_{+} (h^\wedge A),\]
		and
		\[h_{\pm}^\wedge A_{h + 1} \cap h^\wedge A = \max (h_{\pm}^\wedge A_{h + 1}) = \min (h^\wedge A),\]
		it follows that 
		\begin{equation}\label{eq:17}
			\beta \in h_{\pm}^\wedge A_{h + 1} ~\text{and}~ \beta = \max (h_{\pm}^\wedge A_{h + 1}) = \min (h^\wedge A).
		\end{equation}
		Similarly, since 
\[\alpha < \beta = \max (h_{\pm}^\wedge A_{h + 1}) = \min (h^\wedge A)\]
and 
\[\alpha \in h_{\pm}^\wedge A_{h + 1} \cup h^\wedge A,\]
 it follows that
		\begin{equation}\label{eq:18}
			\alpha \in h_{\pm}^\wedge A_{h + 1}.
		\end{equation}
		Thus \eqref{eq:17} and \eqref{eq:18} implies that $$\alpha \in h_{\pm}^\wedge A_{h + 1}~\text{and} ~\beta \in h_{\pm}^\wedge A_{h + 1},$$ which contradicts Claim $1$. Hence
		\[~\text{either}~ \alpha \notin h_{\pm}^\wedge A_{h + 1} \cup h^\wedge A ~\text{or}~ \beta \notin h_{\pm}^\wedge A_{h + 1} \cup h^\wedge A.\]
		Therefore,
		\[~\text{either}~ \alpha \notin C ~\text{or}~ \beta \notin C,\]
which proves Claim $2$.

		Since $C$ is a symmetric set, it also follows that
		\[~\text{either}~ - \alpha \notin C ~\text{or}~ - \beta \notin C.\]
Hence either $- \alpha, \alpha$ or $- \beta, \beta$ are pair of extra elements in $h_{\pm}^{\wedge} A$ distinct from the elements of $C$. Therefore,
		\begin{align*}
			|h_{\pm}^\wedge A| \geq |C| + 2 \geq |h_{\pm}^\wedge A_{h + 1}| + 2h + 2.
		\end{align*}
		Now an application of Lemma \ref{rss-lem2} gives 
		\[|h_{\pm}^\wedge A| \geq h^2 + 2h + 1.\]
		This completes the proof.
	\end{proof}
	
	\begin{lemma}\label{rss-lem7}
		Let $h \geq 5$ be an integer. Let $A = \{a_1, a_2, \ldots, a_{h + 1}\}$ be a set of odd positive integers such that $a_1 < a_2 < \cdots < a_{h + 1}$. If
		\[|h_{\pm}^\wedge A| = (h + 1)^2,\]
		then
		\[A = a_1 \ast \{1, 3, \ldots, 2h + 1\}.\]
	\end{lemma}
	
	\begin{proof}
		By applying Lemma \ref{rss-lem2} and Lemma \ref{rss-lem3}, we get
		\[|h_{\pm}^\wedge A| \geq |h_{\pm}^\wedge A_{h + 1}| + 2h + 2 \geq (h^2 - 1) + 2h + 2 = (h + 1)^2,\]
		where $A_{h + 1} = A \setminus \{a_{h + 1}\}$.
		This implies that
		\[|h_{\pm}^\wedge A_{h + 1}| = h^2 - 1,\]
		and so, it follows from Lemma \ref{rss-lem6} that
		\begin{equation}\label{eq:11}
			A_{h + 1} = a_1 \ast \{1, 3, \ldots, 2h - 1\}. 
		\end{equation}
		Let 
		\[C = h_{\pm}^\wedge A_{h + 1} \cup h^\wedge A \cup h^\wedge (- A),\]
		and let $x, y, z, \alpha$ and $\beta$ be as defined in the proof of Lemma \ref{rss-lem3}. As shown in the proof of Lemma \ref{rss-lem3},
		\[~\text{either}~ \alpha \notin C ~\text{or}~ \beta \notin C.\]
		We show that exactly one of $\alpha$ and $\beta$ does not belong to $C$. Suppose that $\alpha \not \in C$ and $\beta \notin C$.
		Since 
        \[\{- \alpha, - \beta, \alpha, \beta\} \subseteq h_{\pm}^\wedge A \setminus C\]
        and 
        \[|C| = |h_{\pm}^\wedge A_{h + 1}| + 2h = h^2 + 2h - 1,\]
        it follows that
		\[|h_{\pm}^\wedge A| \geq |C| + 4 = (h + 1)^2 + 2,\]
		which is a contradiction. Therefore, exactly one of $\alpha$ and $\beta$ does not belong to $C$.

		\noindent {\textbf{Claim.}} $\beta \in C$.
 
		Suppose that $\beta \notin C$. Since exactly one of $\alpha$ and $\beta$ does not belong to $C$, it follows that  $\alpha \in C$. Since the set $C$ is a symmetric set, it follows that $- \alpha \in C$ also. The first three largest elements of $h_{\pm}^\wedge A_{h + 1}$ are $x, y, z$, respectively. Since
		\[x < \alpha < {\min}_{+}(h^\wedge A),\]
		and
		\[\max (h_{\pm}^\wedge A_{h + 1}) = \min (h^\wedge A),\]
		it follows that 
		\[~\text{either}~ \alpha = y ~\text{or}~ \alpha = z.\]
		Clearly,
		\begin{align*}
			|C \cup \{- \beta, \beta\}|
			& = |h_{\pm}^\wedge A_{h + 1} \cup h^\wedge A \cup h^\wedge (- A)| + |\{- \beta, \beta\}|\\
			&= (h^2 - 1) + (h + 1) + (h + 1) - 2  + 2\\
			&= (h + 1)^2.
		\end{align*}
		
		\noindent {\textbf{Case 1}} ($\alpha = z$). This condition implies that
		\[a_{h + 1} - a_h = 2a_2 = 6a_1.\]
		Let 
		\[v = a_1 + a_2 - a_3 + a_4 + \ldots + a_{h - 1} + a_{h + 1}.\] 
		Then
		\[x < v < y.\]
		Hence
		\[v \not\in h_{\pm}^\wedge A_{h + 1}.\]
		Since 
		\[0 < v < \alpha = \min (h^\wedge A) < \beta,\]
		it follows that
		\[v \notin C \cup \{- \beta, \beta\}.\]
		Thus
        \[v \in h_{\pm}^\wedge A \setminus (C \cup \{- \beta, \beta\}),\]
        and so
        \[|h_{\pm}^\wedge A| \geq |C \cup \{- \beta, \beta\}| + 1 = (h + 1)^2 + 1, \]
		which is a contradiction. Hence
		\[\alpha \neq z.\]
		
		\noindent {\textbf{Case 2}} ($\alpha = y$). This condition implies that
		\[a_{h + 1} - a_h = 2a_2 = 6a_1.\]
		Let 
		\[w = - a_1 - a_2 + a_3 + a_4 + \ldots + a_{h - 1} + a_{h + 1}.\] 
		Then
		\[x < w < y.\]
		Hence
		\[w \not\in h_{\pm}^\wedge A_{h + 1}.\]
		Since 
		\[0 < w < z = \min (h^\wedge A),\]
		it follows that
		\[w \notin C \cup \{- \beta, \beta\}.\]
		Thus
        \[w \in h_{\pm}^\wedge A \setminus (C \cup \{- \beta, \beta\}),\]
        and so
        \[|h_{\pm}^\wedge A| \geq |C \cup \{- \beta, \beta\}| + 1 = (h + 1)^2 + 1, \]
		which is a contradiction. Hence
		\[\alpha \neq y.\]
		Therefore, 
		\[\alpha \neq y ~\text{and}~ \alpha \neq z,\]
		which is a contradiction. Therefore, $\beta \in C$ which proves our claim.
	
		Now, since $\beta \in C$ and $x < y < \beta \leq z = \max (h_{\pm}^\wedge A_{h + 1}) = \min (h^\wedge A )$, it follows that
		\[\beta = z,\]
		and so 
		\[a_{h + 1} - a_h = 2a_1,\]
        which implies
		\begin{equation}\label{eq:12}
			a_{h + 1} = (2h + 1) a_1.
		\end{equation}
		Therefore, it follows from \eqref{eq:11} and \eqref{eq:12} that  
		\[A = a_1 \ast \{1, 3, \ldots, 2h + 1\}.\]
		This completes the proof.
	\end{proof}

The following lemma is a particular case of Theorem $9$ in \cite{mmp2024}. But our proof is different. 
	
	\begin{lemma}\label{rss-lem18}
		Let $A = \{a_1, \ldots, a_5\}$ be a set of odd positive integers such that $a_1 < \cdots < a_5$. Then 
		\[|4_{\pm}^\wedge A| = 25,\]
		if and only if, 
		\[A = a_1 \ast \{1, 3, 5, 7, 9\}.\]
	\end{lemma}
	
\begin{proof} If $A = a_1 \ast \{1, 3, 5, 7, 9\}$, then
       \[|4_{\pm}^\wedge A| = |4_{\pm}^\wedge \{1, 3, 5, 7, 9\}| =25.\]
       
		Now assume that $|4_{\pm}^\wedge A| = 25$. By applying Lemma \ref{rss-lem2} and Lemma \ref{rss-lem3}, we get
		\[|4_{\pm}^\wedge A| \geq |4_{\pm}^\wedge A_5| + 8 + 2 \geq 15 + 10 = 25,\]
		where $A_5 = \{a_1, a_2, a_3, a_4\}$. This implies that 
		\[|4_{\pm}^\wedge A_5| = 15,\]
		and so, it follows from Lemma \ref{rss-lem19} that
		\[\text{either}~ a_4 - a_3 = a_2 - a_1 ~\text{or}~ a_4 - a_3 = a_1 + a_2.\]
		Let 
       \begin{align*}
		 C &= 4_{\pm}^\wedge A_5 \cup 4^\wedge A \cup 4^\wedge (- A),\\
		 \alpha &= a_1 - a_2 + a_3 + a_5 = z + a_5 - a_4 - 2a_2,\\
		 \beta &= - a_1 + a_2 + a_3 + a_5 = z + a_5 - a_4 - 2a_1.	
		\end{align*}
		As shown in the proof of Lemma \ref{rss-lem3}, we have 
		\[~\text{either}~ \alpha \not\in C ~\text{or}~\beta \not\in C.\]  
		Since $|4_{\pm}^\wedge A_5| = 15$, by a similar argument as in the proof of Lemma \ref{rss-lem7}, it can be shown that exactly one of $\alpha$ and $\beta$ does not belong to $C$. Since $|C| = 23$, it follows that either
		\[4_{\pm}^\wedge A = C \cup \{- \alpha, \alpha\} ~\text{with}~ -\alpha, \alpha \not \in C\]
or
\[4_{\pm}^\wedge A = C \cup \{- \beta, \beta\}~\text{with}~ -\beta, \beta \not \in C.\].

Let $x, y, z, w$ and $\mu$ be the elements of $C$ as defined in the the proof of Lemma \ref{rss-lem3}. That is,
		\begin{align*}
			x &= a_1 - a_2 + a_3 + a_4 = z - 2a_2,\\
			y &= - a_1 + a_2 + a_3 + a_4 = z - 2a_1,\\
			z &= a_1 + a_2 + a_3 + a_4,\\
            w &= a_1 + a_2 + a_3 + a_5 = {\min}_{+} (4^\wedge A),\\
            \mu &= a_1 + a_2 + a_4 + a_5.
           \end{align*}
All $23$ elements of $C$ are listed below:
\begin{align*}
		&~-a_2 - a_3 - a_4 - a_5 < - a_1 - a_3 - a_4 - a_5 < - a_1 - a_2 - a_4  - a_5 < - a_1 - a_2 - a_3 - a_5\\
		&< - a_1 - a_2 + a_3 - a_4 < a_1 - a_2 - a_3 - a_4 < - a_1 + a_2 - a_3 - a_4 \\
        &< - a_1 - a_2 + a_3 - a_4, a_1 + a_2 - a_3 - a_4, a_1 - a_2 + a_3 - a_4, - a_1 - a_2 - a_3 + a_4,\\
        & ~~~~~~~~~~~~~~~~~~~~\ - a_1 + a_2 + a_3 - a_4, a_1 - a_2 - a_3 + a_4\\
		&< - a_1 + a_2 - a_3 + a_4 < a_1 + a_2 - a_3 + a_4, - a_1 - a_2 + a_3 + a_4 < x = a_1 - a_2 + a_3 + a_4\\
		&< y = - a_1 + a_2 + a_3 + a_4 < z = a_1 + a_2 + a_3 + a_4 < w = a_1 + a_2 + a_3 + a_5\\ 
        &< \mu = a_1 + a_2 + a_4 + a_5 < a_1 + a_3 + a_4 + a_5 < a_2 + a_3 + a_4 + a_5.
\end{align*}

Let
\[\gamma = - a_1 + a_2 + a_4 + a_5~\text{and}~\lambda = a_1 - a_2 + a_4 + a_5.\]	
		Then
		\[x, y, z, w, \alpha, \beta, \gamma, \lambda, \mu \in 4_{\pm}^\wedge A,\]
and
		\[x < y < z < w < \mu,~ y < \beta,~  x < \alpha < \beta < \gamma < \mu,~ \alpha < \lambda < \gamma, \alpha < w.\]
		Since $0 < \alpha < \beta < \gamma$, and either $4_{\pm}^\wedge A = C \cup \{- \alpha, \alpha\}$ or $4_{\pm}^\wedge A = C \cup \{- \beta, \beta\}$, it follows that $\gamma \in C$. Since $y < \gamma < \mu$ and $C$ can not have any element other than $z$ and $w$ lying between $y$ and $\mu$ it follows that either $\gamma = z$ or $\gamma = w$.

First assume that
\[4_{\pm}^\wedge A = C \cup \{- \beta, \beta\}.\] 
In this case, $\alpha \in C$. Now if $\gamma = z$, then 
		\[a_5 - a_3 = 2a_1.\]
		If $\lambda = \beta$, then 
		\[a_4 - a_3 = 2(a_2 - a_1).\]
		Since $a_4 - a_3 = a_2 - a_1$ or $a_4 - a_3 = a_2 + a_1$, it follows that
		\[\text{either}~ 2(a_2 - a_1) = a_2 - a_1 ~\text{or}~ 2(a_2 - a_1) = a_2 + a_1\]
		which implies that 
        \[\text{either}~ a_1 = a_2 ~\text{or}~ a_2 = 3a_1.\]
		Since $a_2 \neq a_1$, it follows that $a_4 - a_3 = a_1 + a_2$ and $a_2 = 3a_1$. But if $a_4 - a_3 = a_1 + a_2$, then
\[a_5 - a_3 = 2a_1 < a_1 + a_2 = a_4 - a_3,\]
and so $a_5 < a_4$, which is a contradiction. Hence 
		\[\lambda \neq \beta.\]
		Since $\lambda \in 4_{\pm}^\wedge A$, it follows that $\lambda \in C$. Since $x < y < z = \gamma$ and $x < \lambda < \gamma = z$, it follows that 
		\[\lambda = y.\]
		Since $x < \alpha < \lambda = y$, $0 < \alpha < \beta$, it follows that
		\[\alpha \not\in C \cup \{- \beta, \beta\} = 4_{\pm}^\wedge A\]
		which is again a contradiction. Thus $\gamma \neq z$. Hence
		\[\gamma = w\]
		which implies
		\[a_4 - a_3 = 2a_1.\]
		Since $a_4 - a_3 = a_2 - a_1$ or $a_4 - a_3 = a_2 + a_1$, it follows that
		\[\text{either}~ 2a_1 = a_2 - a_1 ~\text{or}~ 2a_1 = a_2 + a_1\]
which implies
\[\text{either}~ a_2 = 3a_1 ~\text{or}~ a_1 = a_2.\]
Since $a_2 \neq a_1$, it follows that $a_4 - a_3 = a_2 - a_1$ and $a_2 = 3a_1$.
	    If $\lambda = \beta$, then 
	    \[a_4 - a_3 = 2(a_2 - a_1)\]
	   which implies
	    \[2(a_2 - a_1) = a_2 - a_1,\]
and so $a_1 = a_2$, which is a contradiction. Hence
	    \[\lambda \neq \beta.\]
	    Since $\lambda \in C \cup \{- \beta, \beta\}$, it follows that $\lambda \in C$.
		Since $x < \alpha < \lambda < \gamma = w$, $x < y < z < w = \gamma$ and $\alpha, \lambda \in C$, it follows that
		\[\alpha = y ~\text{and}~ \lambda = z\]
		which implies that $a_5 - a_4 = 2(a_2 - a_1) = 4a_1$ and $a_5 - a_3 = 2a_2 = 6a_1$. Let 
		\[\theta = a_1 + a_2 - a_3 + a_5 \in 4_{\pm}^{\wedge}A.\]
		Then it is easy to verify that
		\[a_1 + a_2 - a_3 + a_4 < \theta < \alpha = y < \beta\]
		and since $\theta \neq \beta$, it follows that $\theta \in C$. Since $$- a_1 + a_2 - a_3 + a_4 < a_1 + a_2 - a_3 + a_4 < \theta < \alpha = y,$$ and $$- a_1 + a_2 - a_3 + a_4 < a_1 + a_2 - a_3 + a_4, - a_1 - a_2 + a_3 + a_4 < x < y = \alpha,$$ it follows that
		\[\text{either}~\theta = x ~\text{or}~ \theta = - a_1 - a_2 + a_3 + a_4.\]
		If $\theta = - a_1 - a_2 + a_3 + a_4$, then $a_3 = 6a_1$ which is a contradiction because $a_3$ is an odd integer. Hence
		\[\theta = x,\]
		and so
		\[a_3 = 5a_1.\]
		Thus we have
		\[a_2 = 3a_1, a_3 = 5a_1, a_4 = a_3 + 2a_1 = 7a_1, a_5 = a_4 + 4a_1 = 11a_1.\]
		Therefore,
		\[A = a_1 \ast \{1, 3, 5, 7, 11\},\]
and so
		\[|4_{\pm}^\wedge A| = |4_{\pm}^\wedge \{1, 3, 5, 7, 11\}| \geq 26,\]
		which is a contradiction. Hence $\gamma \neq w$, and thus
		\[\gamma \notin \{z, w\}.\]
		which is again a contradiction. Therefore, $4_{\pm}^\wedge A = C \cup \{- \beta, \beta\}$ is not possible. 

Now assume that
		\[4_{\pm}^\wedge A = C \cup \{- \alpha, \alpha\}.\]
		Since $0 < \alpha < \beta$, $\beta \in 4_{\pm}^\wedge A = C \cup \{- \alpha, \alpha\}$ we have
		\[\beta \in C.\]
		If $\lambda = \beta$, then 
		\[a_4 - a_3 = 2(a_2 - a_1).\]
		If $\delta_1 = a_4 - a_3 = a_2 - a_1$ and $a_4 - a_3 = 2(a_2 - a_1)$, it follows that
		\[2(a_2 - a_1) = a_2 - a_1\]
		which implies that $a_1 = a_2$, a contradiction. Hence
		\[\delta_1 = a_4 - a_3 = a_1 + a_2,\]
		and so
		\[a_2 = 3a_1.\]
		If $\gamma = z$, then
		\[a_5 - a_3 = 2a_1.\]
		Since $\delta_1 = a_4 - a_3 = a_1 + a_2$ and $a_5 - a_3 = 2a_1$, it follows that
		\[a_4 - a_3 = a_1 + a_2 > 2a_1 = a_5 - a_3\]
		which implies that $a_4 > a_5$, a contradiction. Hence
		\[\gamma = w,\]
		and so
		\[a_4 - a_3 = 2a_1.\]
		Since $a_4 - a_3 = a_1 + a_2$ and $a_4 - a_3 = 2a_1$, it follows that
		\[a_1 = a_2,\]
		which is contradiction. Hence
		\[\lambda \neq \beta,\]
		which implies that
		\[\lambda \in C.\]
		If $\gamma = z$, then 
		\[a_5 - a_3 = 2a_1.\]
		Since $x < y < z = \gamma$,~ $x < \lambda,~\beta < \gamma = z$,~ $\alpha < \lambda$,~ $\alpha < \beta$ and $\lambda \neq \beta$, it follows that
		\[\text{either}~\beta \notin C \cup \{-\alpha, \alpha\} ~\text{or}~ \lambda \notin C \cup \{-\alpha, \alpha\},\]
		which is a contradiction. Hence $\gamma \neq z$. Since $\gamma \neq z$, it follows that
		\[\gamma = w\]
		and so
		\begin{equation}\label{eq:19}
			a_4 - a_3 = 2a_1.
		\end{equation}
		If $a_4 - a_3 = a_1 + a_2$ and $a_4 - a_3 = 2a_1$, then
		\[a_1 = a_2\]
		which is contradiction. Hence 
		\[a_4 - a_3 = a_2 - a_1\]
		which implies that
		\begin{equation}\label{eq:20}
			a_2 = 3a_1.
		\end{equation}
This relation also implies that $\lambda < \beta$. Since $x < y < z < w = \gamma$, $x < \lambda < \beta < \gamma = w$, and $\lambda, \beta \in C$, it follows that
		\[\lambda = y~\text{and}~ \beta = z.\]
Both of this equalities implies thta
		\begin{equation}\label{eq:21}
			a_5 - a_4 = 2a_1.
		\end{equation}
        Let 
		\[\theta = a_1 + a_2 - a_3 + a_5.\]
		Then it is easy to verify that
		\[a_1 + a_2 - a_3 + a_4 < \theta < \lambda = y\]
		and so 
		\[\theta \in C.\]
		Since 
\[- a_1 + a_2 - a_3 + a_4 < a_1 + a_2 - a_3 + a_4 < \theta < \lambda = y,\]
and
\[- a_1 + a_2 - a_3 + a_4 < a_1 + a_2 - a_3 + a_4, - a_1 - a_2 + a_3 + a_4 < x < y = \lambda,\]
it follows that either
		\[\theta = x ~\text{or}~ \theta = - a_1 - a_2 + a_3 + a_4.\]
		If $\theta = x$, then $a_3 = 4a_1$, which is a contradiction because $a_3$ is an odd integer. Hence $\theta = - a_1 - a_2 + a_3 + a_4$, which implies
		\begin{equation}\label{eq:22}
			a_3 = 5a_1.
		\end{equation} 
		Hence it follows from \eqref{eq:19}, \eqref{eq:20}, \eqref{eq:21} and \eqref{eq:22} that
		
		\[a_2 = 3a_1, a_3 = 5a_1, a_4 = a_3 + 2a_1 = 7a_1, a_5 = a_4 + 2a_1 = 9a_1.\]
		Therefore,
		\[A = a_1 \ast \{1, 3, 5, 7, 9\}.\]
		This completes the proof.
	\end{proof}

The following lemma is the part of the result contained in \cite{mmp2024} which will be useful to prove Theorem \ref{rssp-thm-2}.

\begin{lemma}[{\cite[Theorem $5$]{mmp2024}}]\label{rssp-inv-lem}
Let $h$ and $k$ be positive integers such that $4 \leq h \leq k - 1$. Let $A = \{a_1, a_2,  \ldots, a_k\}$ be a set of positive integers with $a_1 < a_2 < \cdots < a_k$ such that
\[|h_{\pm}^\wedge A| = 2hk - h^2 + 1.\]
Let the set $B = \{a_1, a_2,  \ldots, a_{h + 1}\} \subseteq A$ be in arithmetic progression. Then
\[A = a_1 \ast \{1, 3, \ldots, 2k - 1\}.\]
\end{lemma}

\section{Proof of Theorem \ref{rssp-thm-1} and Theorem \ref{rssp-thm-2}}\label{section-rrsp-thm-proof}

\begin{proof}[Proof of Theorem \ref{rssp-thm-3}]
Let $A = \{a_1, \ldots, a_{h + 1}\}$, where $a_1 < \cdots < a_{h+1}$. Since 
\[h_{\pm}^\wedge (d \ast A)| = |d \ast h_{\pm}^\wedge A| = |h_{\pm}^\wedge A|\]
for all positive integers $d$, we may assume that 
\[\gcd(a_1, \ldots, a_{h+1}) = 1.\]
The theorem holds for $h = 3$ as proved in \cite{bhanja-kom-pandey2021}. Therefore, we assume that $h \geq 4$. 

\noindent {\textbf{Case 1}} ($a_2 \equiv a_1 \pmod 2$ and $a_r \not \equiv a_1 \pmod 2$ for  some $r \in [3, h + 1]$). In this case, it follows from Lemma \ref{rss-lem1} that
	\[|h_{\pm}^\wedge A| \geq h^2 + 3h + 2 > h^2 + 2h + 1.\]

\noindent {\textbf{Case 2}} ($a_1 \equiv a_2 \equiv \cdots \equiv a_{h + 1} \pmod 2$). In this case, if $a_1 \equiv 0 \pmod 2$, then 
\[a_1 \equiv a_2 \equiv \cdots \equiv a_{h + 1} \equiv 0 \pmod 2.\]
This implies that $\gcd(a_1, \ldots, a_{h+1}) \geq 2$, which is a contradiction. Hence $a_1 \not \equiv 0 \pmod 2$. Therefore, Lemma \ref{rss-lem3} implies that 
\[|h_{\pm}^\wedge A| \geq h^2 + 2h + 1.\]
	
\noindent {\textbf{Case 3}} ($a_2 \not \equiv a_1 \pmod 2$). In this case, if $a_3 \not \equiv a_1 \pmod 2$, then Lemma \ref{rss-lem4} implies that 
	\begin{equation*}
		|h_{\pm}^\wedge A| \geq 
		\begin{cases}
			h^2 + 3h, &\mbox{if }  a_3 = 2a_1 + a_2;\\
			h^2 + 4h - 1, &\mbox{if } a_3 \neq 2a_1 + a_2.   
		\end{cases}
	\end{equation*}
	Hence
	\[|h_{\pm}^\wedge A| \geq h^2 + 2h + 2 > h^2 + 2h + 1.\]	
If $a_3 \equiv a_1 \pmod 2$, then Lemma \ref{rss-lem5} implies that 
	\begin{equation*}
		|h_{\pm}^\wedge A| \geq 
		\begin{cases}
			h^2 + 2h + 2, &\mbox{if } A_2 ~\text{is not in an A.P. and}~ h \geq 4;\\
			\frac{1}{2} h(3h - 1) + 4, &\mbox{if } h \geq 4, A_2 ~\text{is an A.P. and}~ a_2 \not \equiv 0 \pmod 2;\\
			26, &\mbox{if } A_2 ~\text{is an A.P. and}~ a_2 \equiv 0 \pmod 2, h = 4;\\
			2h(h - 1), &\mbox{if } A_2 ~\text{is an A.P. and}~ a_2 \equiv 0 \pmod 2, h \geq 5.  
		\end{cases}
	\end{equation*}
	Hence
	\[|h_{\pm}^\wedge A| \geq h^2 + 2h + 2 > h^2 + 2h + 1.\]	
Thus, in all cases, we have
\[|h_{\pm}^\wedge A| \geq h^2 + 2h + 1.\]
To show that the lower bound in $\eqref{rssp-thm-3eq1}$ is best possible, consider the set $A = \{1, 3, \ldots, 2h-1\}$. Then 
		\[h_{\pm}^\wedge A \subseteq [-h^2 - 2h, h^2 + 2h].\]
If $h$ is even, then $h_{\pm}^\wedge A$ contains only even integers in the above interval. Hence
		\[|h_{\pm}^\wedge A| \leq  2h^2 + 4h + 1 - (h^2 + 2h) = h^2 + 2h + 1.\]
But 
		\[|h_{\pm}^\wedge A| \geq  h^2 + 2h + 1,\]
		and so
		\[|h_{\pm}^\wedge A| =  h^2 + 2h + 1.\]
		Similarly, if $h$ is odd, then
		\[|h_{\pm}^\wedge A| =  h^2 + 2h + 1.\]
		This completes the proof.
\end{proof}

\begin{proof}[Proof of Theorem \ref{rssp-thm-1}]
The theorem easily follows from Lemma \ref{rssp-basic-lem1} and Theorem \ref{rssp-thm-3}.
To show that the lower bound in $\eqref{rssp-thm-1eq1}$ is best possible, let 
		\[A = \{1, 3, \ldots, 2k - 1\}.\]
		Then 
		\[\min (h_{\pm}^\wedge A) = - (2k - 2h + 1) - (2k - 2h + 3) - \cdots - (2k - 1) = - 2hk + h^2,\]
		and
		\[\max (h_{\pm}^\wedge A) = (2k - 2h + 1) + (2k - 2h + 3) + \cdots + (2k - 1) = 2hk - h^2.\]
		Hence
		\[h_{\pm}^\wedge A \subseteq [- 2hk + h^2, 2hk - h^2].\]
		It is easy to verify that
		\[h \equiv 2hk - h^2 \equiv x \pmod 2\]
		for each $x \in h_{\pm}^\wedge A$. 
		
		Now if $h$ is an even integer and $A$ contains only odd integers, then $h_{\pm}^\wedge A$ does not contains any odd integers, hence
		\[|h_{\pm}^\wedge A| \leq  2hk - h^2 + 1,\]
		but 
		\[|h_{\pm}^\wedge A| \geq  2hk - h^2 + 1,\]
		and so
		\[|h_{\pm}^\wedge A| =  2hk - h^2 + 1.\]
		Similarly, if $h$ is odd integer, then
		\[|h_{\pm}^\wedge A| =  2hk - h^2 + 1.\]
		This completes the proof.
\end{proof}
	
\begin{proof}[Proof of Theorem \ref{rssp-thm-2}]
The theorem holds for $h = 3$ as proved in \cite{bhanja-kom-pandey2021}. Therefore, we assume that $h \geq 4$. 
Let $A = \{a_1, \ldots, a_k\}$, where $a_1 < \cdots < a_k$. Let $B = \{a_1, \ldots, a_{h + 1}\} \subseteq A$, and let $A' = A \setminus \{a_1\}$. Then
\[(-h^{\wedge} A') \cup h_{\pm}^{\wedge} B \cup h^{\wedge} A' \subseteq  h_{\pm}^{\wedge} B.\]
Since
\[(-h^{\wedge} A') \cap h_{\pm}^{\wedge} B = \{-a_2 - \cdots - a_{h + 1}\},\]
and
\[h_{\pm}^{\wedge} B \cap h^{\wedge} A' = \{a_2 + \cdots + a_{h + 1}\},\]
it follows that
\[|h_{\pm}^{\wedge} A| \geq |h_{\pm}^{\wedge} B| + 2 |h^{\wedge} A'| - 2.\]
Therefore, applying Theorem \ref{restricted-hfold-direct-thm} and Theorem \ref{rssp-thm-3}, we get
\begin{align*}
  2hk - h^2 + 1 = |h_{\pm}^{\wedge} A| & \geq |h_{\pm}^{\wedge} B| + 2 |h^{\wedge} A'| - 2\\
   & \geq |h_{\pm}^{\wedge} B| + 2(h(k-1) - h^2 + 1) - 2 \\
   & \geq (h^2 + 2h + 1) + 2(h(k-1) - h^2 + 1) - 2\\
   & = 2hk -h^2 + 1.
\end{align*}
The above inequalities imply that
\[|h_{\pm}^{\wedge} B| = h^2 + 2h + 1.\]
Now let $d = \gcd (a_1, \ldots, a_{h + 1})$, and let $a_i' = a_i/d$ for $i \in [1, h + 1]$. Then $\gcd (a_1', \ldots, a_{h + 1}') = 1$. Let $B' = \{a_1', \ldots, a_{h + 1}'\}$. Then $B = d \ast B'$, and so
\[|h_{\pm}^{\wedge} B'| = |h_{\pm}^{\wedge} B| = h^2 + 2h + 1.\]
It follows from the proof of Theorem \ref{rssp-thm-3} that the above inequality holds if and only if all the elements of $B'$ are odd positive integers. Therefore, Lemma \ref{rss-lem7} and Lemma \ref{rss-lem18} imply that
\[B' = a_1' \ast \{1, 3, \ldots, 2h + 1\},\]
and so
\[B = d \ast B' = da_1' \ast \{1, 3, \ldots, 2h + 1\} = a_1 \ast \{1, 3, \ldots, 2h + 1\}.\]
Now Lemma \ref{rssp-inv-lem} implies that
\[A = a_1 \ast \{1, 3, \ldots, 2k - 1\}.\]
This completes the proof.
\end{proof}	
	
	

\end{document}